\documentclass{article}

\PassOptionsToPackage{numbers, compress}{natbib}

\usepackage[preprint]{neurips_2020}

\usepackage[utf8]{inputenc} %
\usepackage[T1]{fontenc}    %
\usepackage[hidelinks]{hyperref}       %
\usepackage{url}            %
\usepackage{booktabs}       %
\usepackage{amsfonts}       %
\usepackage{nicefrac}       %
\usepackage{microtype}      %

\title{A generic adaptive restart scheme with applications to saddle point algorithms}

\author{%
  Oliver Hinder \\
  Google Research, University of Pittsburgh \\
  \texttt{ohinder@pitt.edu} \\
  \And
  Miles Lubin \\
  Google Research \\
  \texttt{mlubin@google.com} \\
}

\usepackage{geometry}                		%
\usepackage{graphicx}				%
\usepackage{amssymb}
\usepackage{amsthm}
\usepackage{url}
\usepackage{algorithm2e}
\usepackage{enumerate}
\usepackage{mathtools}

\newcommand{\R}{{\bf R}}
\newcommand{\N}{{\bf N}}
\newcommand{\e}{{e}}

\newcommand{\argmin}{\mathop{\rm argmin}}

\newcommand{\argmax}{\mathop{\rm argmax}}

\newcommand{\sign}{\mathop{\bf sign}}

\newcommand{\ball}[2]{B_{#1}(#2)}

\newcommand{\backtrackingfactor}{\eta}
\usepackage{nicefrac}

\usepackage{physics}

\newtheorem{theorem}{Theorem}
\newtheorem*{theorem*}{Theorem}
\newtheorem{lemma}{Lemma}
\newtheorem{fact}{Fact}
\newtheorem{definition}{Definition}

\newtheorem{remark}{Remark}

\newtheorem{corollary}{Corollary}
\newtheorem{assumption}{Assumption}

\usepackage{todonotes}

\SetKwFunction{GenericAdaptiveRestartScheme}{GenericAdaptiveRestartScheme}
\newcommand{\callGenericAdaptiveRestartScheme}[1]{\hyperref[function:GenericAdaptiveRestartScheme]{\GenericAdaptiveRestartScheme}{#1}}

\SetKwFunction{GenericIterativeAlgorithm}{GenericIterativeAlgorithm}

\newcommand{\callGenericIterativeAlgorithm}[1]{\hyperref[function:GenericIterativeAlgorithm]{\GenericIterativeAlgorithm{#1}}}

\SetKwFunction{OneStepOfAlgorithm}{OneStepOfAlgorithm}
\SetKwFunction{InitializeAlgorithm}{InitializeAlgorithm}

\SetKwFunction{OneStepOfPDHG}{OneStepOfPDHG}
\SetKwFunction{InitializePDHG}{InitializePDHG}
\SetKwFunction{AdaptiveRestartPDHG}{AdaptiveRestartPDHG}

\newcommand{\callAdaptiveRestartPDHG}[1]{\hyperref[function:AdaptiveRestartPDHG]{\AdaptiveRestartPDHG}{#1}}

\SetKwFunction{OneStepOfAGD}{OneStepOfAGD}
\SetKwFunction{InitializeAGD}{InitializeAGD}
\SetKwFunction{AdaptiveRestartAGD}{AdaptiveRestartAGD}
\newcommand{\callAdaptiveRestartAGD}[1]{\hyperref[function:AdaptiveRestartAGD]{\AdaptiveRestartAGD}{#1}}

\SetKwFunction{OneStepOfExtragradient}{OneStepOfExtragradient}
\SetKwFunction{InitializeExtragradient}{InitializeExtragradient}
\SetKwFunction{AdaptiveRestartExtragradient}{AdaptiveRestartExtragradient}

\newcommand{\callAdaptiveRestartExtragradient}[1]{\hyperref[function:AdaptiveRestartExtragradient]{\AdaptiveRestartExtragradient}{#1}}

\SetKwProg{Fn}{Function}{:}{}
  
\newcommand{\SmoothFunction}{{a}}  \newcommand{\ProxFunction}{{b}}

\newcommand{\PdhgStepSize}{{\gamma}} 
\newcommand{\PrimalStepSize}{{\gamma_x}}  \newcommand{\DualStepSize}{{\gamma_y}}

\newcommand{\InitialTau}[1]{\tau_{\text{init}}}

\newcommand{\outerIterate}[0]{v}

\begin{document}

\maketitle
\begin{abstract}
    We provide a simple and generic adaptive restart scheme for convex optimization that is able to achieve worst-case bounds matching (up to constant multiplicative factors) optimal restart schemes that require knowledge of problem specific constants. 
    The scheme triggers restarts whenever there is sufficient reduction of a distance-based potential function.
    This potential function is always computable.
    
    We apply the scheme to obtain the first adaptive restart algorithm for saddle-point algorithms including primal-dual hybrid gradient (PDHG) and extragradient.
    The method improves the worst-case bounds of PDHG on bilinear games, and numerical experiments on quadratic assignment problems and matrix games demonstrate dramatic improvements for obtaining high-accuracy solutions.
    Additionally, for accelerated gradient descent (AGD), this scheme obtains a worst-case bound within 60\% of the bound achieved by the (unknown) optimal restart period when high accuracy is desired.
    In practice, the scheme is competitive with the heuristic of O'Donoghue and Candes \cite{o2015adaptive}.
    
\end{abstract}

\section{Introduction}

Periodically re-initializing an algorithm from its output, known as restarting, is a popular method for generically speeding up optimization algorithms.
More precisely, given an algorithm with a sublinear convergence rate and an additional regularity assumption (e.g., strong convexity) periodically restarting the algorithm can improve its convergence from sublinear to linear \cite{nemirovskii1985optimal}.
The difficulty is in choosing the length of the restart period.
One can resolve this issue by running a grid search on a log-scale to find the best restart period \cite{roulet2017sharpness}, but this can be many times more computationally intensive than running with a single restart period.
Furthermore, a restart interval length that changes as the algorithm runs, adapting to the local curvature of a function, may outperform a fixed restart period.
Addressing these two issues motivates the study of adaptive restart schemes.

Adaptive restarting was pioneered by \citet{o2015adaptive} for Nesterov's accelerated gradient descent (AGD) \cite{nesterov1983method}. Their paper suggests restarting AGD whenever the function value increases. They observe that this heuristic is competitive with selecting the best restart period via hyperparameter search. They also provide a theoretical analysis showing that for quadratics, as the number of iterations tends to infinity, the worst-case performance of their method is almost optimal. Despite its success in practice, one drawback of O'Donoghue and Candes's \cite{o2015adaptive} restart scheme is that its theory only applies to quadratics. A large number of papers have followed this work by trying to develop restart schemes that avoid hyperparameter searches while obtaining stronger guarantees. These schemes either miss the optimal bound by log factors \cite{lin2014adaptive,nesterov2013gradient,liu2017adaptive}, or are based on maintaining a monotonely increasing minimum restart interval length which may limit adaptivity \cite{alamo2019restart, fercoq2019adaptive}.

Furthermore, there is little study of restart schemes in the context of saddle point algorithms.
\citet{chambolle2011first} show that a restart scheme can be used to improve the sublinear converge rate of primal-dual hybrid gradient (PDHG) under a uniform convexity assumption. \citet{zhao2020optimal} explores restart schemes for stochastic PDHG.
However, these papers neither offer adaptive schemes nor provide linear convergence guarantees.
There are many papers showing the linear convergence of saddle point algorithms without restarts \cite{tseng1995linear,hong2017linear,eckstein1989splitting,boley2013local,lu2020s}. For example, in the bilinear setup it is known that the last iterate of PDHG  \cite{du2018linear} and extragradient \cite{mokhtari2019unified,zhangconvergence} converge linearly to the optimal solution.
However, all these results obtain suboptimal convergence guarantees, even in the simple bilinear setup.

\paragraph{Our contributions}
\begin{enumerate}
    \item We provide a simple adaptive restart scheme where restarts are triggered using a distance based potential function. In a general setup we precisely characterize the linear convergence rate of this scheme.
    \item We apply this scheme to specific algorithms. For PDHG it obtains the optimal convergence rate for bilinear games. To the best of our knowledge this provides the first example of this rate being achieved by a saddle point algorithm. For AGD on smooth strongly convex functions it yields the optimal convergence rate up to small constant factors.
    \item We run experiments with this restart scheme. For PDHG, on matrix games and linear programming relaxations of quadratic assignment problems, our scheme dramatically improves on the ability of ``vanilla'' PDHG to find high accuracy solutions. For AGD, on regularized logistic regression and LASSO problems, our scheme is competitive with the scheme of O'Donoghue and Candes \cite{o2015adaptive}. 
\end{enumerate}

 \paragraph{Notation} 
 Let $\R$ denote the set of real numbers, $\R^{+}$ the positive real numbers, and $\N$ the set of natural numbers (starting from one).
 Let $\ln(\cdot)$ and $\e$ refer to the natural log and exponential respectively. 
 Let $\lambda_{\min}( \cdot )$ denote the minimum eigenvalue of a matrix. 
 Let $X \subseteq \R^{n}$ and $Y \subseteq \R^{m}$ for $n, m \in \N$ be closed convex sets and denote $W = X \times Y$ where for all $w \in W$ we denote $w_x \in X$ as the primal component and $w_y \in Y$ as the dual component.
 Let $\| \cdot \|$ be an arbitrary norm on elements of $W$ and define $\ball{r}{w} := \{ \hat{w} \in W : \| w - \hat{w} \| \le r \}$. If $\hat{W}$ is a set then we define $\| \hat{W} - w \| = \inf_{\hat{w} \in W} \| \hat{w} - w \|$.

\section{Our adaptive restart scheme}

\begin{algorithm}[htpb]\label{function:GenericAdaptiveRestartScheme}
 \Fn{\GenericAdaptiveRestartScheme{$\outerIterate_0$}}{
  \For{$i = 1, \dots, \infty$}{
 $w_{i}^{0}, s_{i}^{0} \gets $ \InitializeAlgorithm{$\outerIterate_{i-1}$} \;
$t \gets 0$\;
\Repeat{restart condition \eqref{eq:our-restart-condition} holds}{
$t \gets t + 1$\;
$w_{i}^{t}, s_{i}^{t} \gets$ \OneStepOfAlgorithm{$w_{i}^{t-1}, s_{i}^{t-1}$} \;
}
$\tau_i \gets t$, $\outerIterate_i \gets w_{i}^{t}$ \;
}
}
\end{algorithm}

\callGenericAdaptiveRestartScheme{} presents our restart scheme. We use $i$ to index the restart epochs, $\tau_i$ to represent the length of the $i$th restart epoch, and $\outerIterate_i$ to be the point produced by the algorithm at the end of the $i$th restart epoch. Within each restart epoch $i$, $t$ keeps track of the number of iterations. The object $s_{i}^{t}$ is the internal state of the algorithm. For example, for AGD it would include momentum information; for PDHG it contains the current iterate and $t$. The vector $w_{i}^{t}$ is what the algorithm outputs at iteration $t$ to obtain its sublinear convergence guarantees. For example, for primal dual hybrid gradient this would be the running average of the iterates since the last restart.

\paragraph{Our restart condition.} Let $\beta \in (0,1)$ be an instance-independent parameter (e.g., $\beta = 1/2$) and $\phi : \R \rightarrow \R$ a instance-independent function. 
For $i > 1$, restart if
\begin{flalign}\label{eq:our-restart-condition}
\frac{\| w^{t}_i - \outerIterate_{i-1} \|}{\phi(t)} \le 
\beta \frac{\| \outerIterate_{i-1} - \outerIterate_{i-2} \|}{\phi(\tau_{i-1})};
\end{flalign}
for $i = 1$, for simplicity we suggest restarting when $t = 1$, but for the analysis we allow the initial restart length $\tau_1$ to be arbitrarily chosen.
As detailed in Section~\ref{sec:theory-to-specific}, the function $\phi$ is selected to correspond to the sublinear convergence bound of the algorithm. For example, AGD convergences proportional to $1/(t+1)^2$, and therefore $\phi(t) := (t+1)^2$. Similarly, PDHG converges proportional to $1/t$, and therefore $\phi(t) := t$.
If $\| w^{t}_i - \outerIterate_{i-1} \|$ is bounded and
$\phi(t) \rightarrow \infty$ as 
$t \rightarrow \infty$ then for $t$ sufficiently large a restart will be triggered.

There are many other possible restart conditions found in literature. For example, a popular restart scheme for AGD is to restart when $f(w_i^{t}) - \inf_{w \in W} f(w) \le \beta \left(f(\outerIterate_{i-1}) - \inf_{w \in W} f(w) \right)$.
In other words, when the potential function $f(w_i^{t}) - \inf_{w \in W} f(w)$ is reduced by a constant factor.
Unfortunately, this assumes knowledge of the minimum of $f$, which is rarely available. Our restart scheme can also be viewed as decreasing the potential function 
\begin{flalign}\label{eq:potential-function}
\Upsilon(\outerIterate_{i-1}, t) := \frac{\| w_{i}^t - \outerIterate_{i-1} \|}{\phi(t)},
\end{flalign}
by at least a $\beta$-factor each iteration.
However, the benefit of \eqref{eq:potential-function} over $f(w_i^t) - \inf_{w \in W} f(w)$ is that it is always computable.
The idea behind choosing this as a potential function is that for many algorithms as \eqref{eq:potential-function} decreases the distance to optimality decreases. This is made explicit later in Lemma~\ref{lemma:bound-r-star-new-by-r-old}.

\section{Assumptions required for adaptive restarts}\label{assume:required-for-adaptive}

This section defines a broad set of conditions under which we will be able to adaptively restart an algorithm to obtain linear convergence. The main property we require is that \emph{without} restarts the algorithm reduces the \emph{localized duality gap} which is defined as
\begin{flalign}\label{def:Delta-r}
\Delta_r(w) := \sup_{(x,y) \in \ball{r}{w}} f(w_x, y) - f(x, w_y)
\end{flalign}
where $r \in \R^{+}$, $w_x$ denotes the primal iterate and $w_y$ the dual iterate. We prepend the word `localized' to `duality gap' because the radius $r$ may be significantly smaller than the distance to optimality. If the radius is larger than the distance to optimality then it represents a valid duality gap. Let $W^{*} := \{ w \in W : \Delta_r(w) = 0, \forall r \in \R^{+} \}$ be the set of optimal solutions.
Note if we are interested in only minimizing a function then we can let $W = X$ and \eqref{def:Delta-r} simplifies to
$\Delta_r(w) = f(w) - \inf_{\tilde{w} \in \ball{r}{w}} f(\tilde{w})$.

\begin{assumption}\label{assume-Delta-r-marginal-gain}
Suppose that for all $w \in W$ and $r_a \in (0,\infty), r_b \in [0,\infty)$ the inequality
$\Delta_{r_b}(w) \le \max\{1, r_b / r_a \} \Delta_{r_a}(w)$ holds.
\end{assumption}

Assumption~\ref{assume-Delta-r-marginal-gain} is extremely mild, and as Lemma~\ref{lem:Delta-r-marginal-gain} demonstrates it suffices for $f$ to be convex-concave for it to hold. The proof of Lemma~\ref{lem:Delta-r-marginal-gain} appears in Appendix~\ref{sec:proof-of:lem:Delta-r-marginal-gain}.

\begin{lemma}
\label{lem:Delta-r-marginal-gain}
Let $X$ and $Y$ be closed convex sets.
Also, suppose $f(x,y)$ is continuous, convex in $x$ for all $y \in Y$, and concave in $y$ for all $x \in X$ then Assumption~\ref{assume-Delta-r-marginal-gain} holds.
\end{lemma}

\begin{algorithm}[htpb]\label{function:GenericIterativeAlgorithm}
  \Fn{\GenericIterativeAlgorithm{$\outerIterate$}}{
 $w^0, s^0 \gets $ \InitializeAlgorithm{$\outerIterate$} \;
\For{$t = 1, \dots, \infty$}{
$w^{t}, s^{t} \gets$ \OneStepOfAlgorithm{$w^{t-1}, s^{t-1}$} \;
}
}
\end{algorithm}

Consider a generic iterative algorithm described in \callGenericIterativeAlgorithm{$\outerIterate$}. To prove our results we need $\Delta_r(w)$ to be reduced in a predictable manner as per Assumption~\ref{assume:reduce-potential-function}.

\begin{assumption}\label{assume:reduce-potential-function}
Consider the sequence $\{w^t\}_{t=0}^{\infty}$ generated by \callGenericIterativeAlgorithm{$\outerIterate$}. 
We assume there exists a (known) constant $\beta \in (0,1)$ and a (known) function $\phi : \R^{+} \rightarrow \R^{+}$ that is strictly increasing log-concave and unbounded from above. Furthermore assume there exists an (unknown)  constant $C \in (0, \infty)$ such that
for all $\outerIterate \in W$ and $t \in \N$, $\Delta_{\beta r} (w^t) \le
\frac{C (1 + \beta)^2 r^2 }{\phi(t)}$ 
where $r = \| w^{t} - \outerIterate \|$. 
\end{assumption}

Assumption~\ref{assume:reduce-potential-function} is strongly related to the standard sublinear convergence bound for an algorithm. For example, as we show in Section~\ref{sec:theory-to-specific}, for AGD $\beta$ takes any value in $(0,1)$, $C = 2 L$ and $\phi(t) = (t+1)^2$ where $L$ is the smoothness constant of $f$. We note that if $\phi$ takes the form $\phi(k) = (k+c_1)^{c_2}$ for constants $c_1 \ge 0$ and $c_2 > 0$ then $\phi$ is a strictly increasing log concave function.

\begin{assumption}
\label{assume:contraction}
Consider the sequence $\{w^t\}_{t=0}^{\infty}$ generated by \callGenericIterativeAlgorithm{$\outerIterate$}. 
Suppose that there exists some monotone decreasing function $Q : \R \rightarrow \R^{+}$ that satisfies $\| w^{t} - w^{*} \| \le Q(t) \| \outerIterate - w^{*} \|$ for all $w^{*} \in W^{*}$, $t \in \N$ and $\outerIterate \in W$.
\end{assumption}

The choice of the function $Q$ in Assumption~\ref{assume:contraction} affects the final linear convergence bound. To achieve the tightest linear convergence bound one should aim to select the smallest (valid) $Q(t)$ possible. It may even be a constant as is the case for PDHG.
We further need the localized duality gap to satisfy an error bound property, this is the purpose of Assumption~\ref{assume:error-bound}. 

\begin{assumption}\label{assume:error-bound}
There exists some constant $\theta \in (0,\infty)$ such that for all $w \in W$, if $r^{*} = \| W^{*} - w \|$ then $\theta (r^{*})^2 \le \Delta_{r^{*}}(w)$.
\end{assumption}

\begin{remark}\label{rem:quadratic-growth}
Suppose that $W = X$ and $f(w) \ge \frac{\alpha}{2} \| w - W^{*} \|^2 + \inf_{\hat{w} \in W} f(\hat{w})$ then Assumption~\ref{assume:error-bound} holds with $\theta = \alpha / 2$. This condition is known as quadratic growth and is weaker than $\alpha$-strong convexity of $f$.
\end{remark}

\newcommand{\sigmaMin}[0]{\sigma}

\begin{lemma}\label{lem:minimum-nonsingular-value}
Suppose $W = \R^{n} \times \R^{m}$ for $n, m \in \N$, and that $f(x,y) = c^T x + y^T A x + b^T y$ has a saddle point $(x^{*}, y^{*})$. Let $\sigmaMin{}$ denote the minimum nonzero singular value of $A$.
Then Assumption~\ref{assume:error-bound} holds where $\| \cdot \|$ is the Euclidean norm and $\theta = \sigmaMin{}$.
\end{lemma}

The proof of Lemma~\ref{lem:minimum-nonsingular-value} appears in Section~\ref{sec:proof-of:lem:minimum-nonsingular-value}. Error bounds conditions are a well-studied topic \cite{hoffman1952approximate,lojasiewicz1961probleme,drusvyatskiy2018error,fabian2010error}.
We feel our condition is natural since it captures both strongly convex functions and bilinear games, two important areas where linear convergence is possible.

\section{Proof of convergence for generic restart scheme}\label{sec:proof-of-convergence}

For the following discussions \textbf{we define the condition number} $\hat{\kappa} := \frac{C}{\theta}$,
where the constants $C$ and $\theta$ are those from Assumption~\ref{assume:reduce-potential-function} and \ref{assume:error-bound} respectively.

Lemma~\ref{lemma:bound-r-star-new-by-r-old} establishes if Assumptions~\ref{assume-Delta-r-marginal-gain}, \ref{assume:reduce-potential-function}, and \ref{assume:error-bound} hold then \eqref{eq:potential-function} is a valid potential function in the sense that for fixed $t$ if $\Upsilon(\outerIterate_{i-1}, t) = \frac{\| w^t_i - \outerIterate_{i-1} \|}{\phi(t)} \rightarrow 0$ as $i \rightarrow \infty$ then $\| W^{*} - w_{i}^t \| \rightarrow 0$.

\begin{lemma}\label{lemma:bound-r-star-new-by-r-old}
Consider the sequence $\{w^t\}_{t=0}^{\infty}$ generated by \callGenericIterativeAlgorithm{$\outerIterate$} for $\outerIterate \in W$.
Suppose $t \in \N$, and that Assumptions~\ref{assume-Delta-r-marginal-gain}, \ref{assume:reduce-potential-function}, and \ref{assume:error-bound} hold. Then,
\begin{flalign*}
\| W^{*} - w^{t} \| &\le
\begin{cases}
\frac{(1 + \beta)^2}{\beta} \frac{\hat{\kappa}  }{\phi(t)} \| w^{t} - \outerIterate \| & \| W^{*} - w^{t} \| \ge \beta \| w^{t} - \outerIterate \|  \\
(1 + \beta) \sqrt{ \frac{\hat{\kappa}}{\phi(t)} } \| w^{t} -  \outerIterate \| & \text{otherwise.}
\end{cases}
\end{flalign*}
\end{lemma}

\begin{proof}
We obtain,
\begin{flalign*}
\theta \| W^{*} - w^{t} \|^{2} &\le \Delta_{\| W^{*} - w^{t} \|} (w^t) & \text{by Assumption~\ref{assume:error-bound},} \\
&\le \max\left\{ 1, \frac{\| W^{*} - w^{t} \|}{\beta \| w^{t} - \outerIterate \|} \right\} \Delta_{\beta \| w^{t} - \outerIterate \|}(w^t) &  \text{by Assumption~\ref{assume-Delta-r-marginal-gain},} \\
&\le  \max\left\{ 1, \frac{\| W^{*} - w^{t} \|}{\beta \| w^{t} - \outerIterate \|} \right\} \frac{C (1 + \beta)^2 \| w^{t} - \outerIterate \|^{2}}{\phi(t)} & \text{ by Assumption~\ref{assume:reduce-potential-function}}.
\end{flalign*}
Rearranging to bound $\| W^{*} - w^{t} \|$ gives the result.
\end{proof}

Define $t^{*} \in [2, \infty)$ as any solution to
\begin{flalign}
\label{define:t-star}
\frac{\phi(t^{*} - 2)}{(1 + Q(t^{*}-2))^2} \ge \frac{ (1 + \beta)^2}{\beta^2} \hat{\kappa}.
\end{flalign}
Note a solution to  \eqref{define:t-star} must exist because $\phi(t)$ is unbounded from above and $Q(t)$ monotone decreasing.
The quantity $t^{*}$ is useful for summarizing our results. 
In particular, as detailed in the proof of Theorem~\ref{thm:main-result}, $t^{*}$ represents an upper bound on the restart interval length given $\tau_1 \le t^{*}$.
But first we prove Lemma~\ref{lem:decrease-distance-potential-function-advanced} which makes a statement on when restarts are triggered. 
The proof of Lemma~\ref{lem:decrease-distance-potential-function-advanced} appears in Section~\ref{app:proof-of:lem:decrease-distance-potential-function-advanced}.

\begin{lemma}\label{lem:decrease-distance-potential-function-advanced}
Consider \callGenericAdaptiveRestartScheme{}.
Suppose that Assumptions~\ref{assume-Delta-r-marginal-gain}, \ref{assume:reduce-potential-function}, \ref{assume:contraction}, and~\ref{assume:error-bound} hold.
If $t \in \N$ is such that $\phi(t) \ge \max\{ \sqrt{ \phi(t^{*}-2) \phi(\tau_{i-1}) }, \phi(t^{*}-2) \}$ then 
$\frac{\| w_i^t - \outerIterate_{i-1} \|}{\phi(t)} \le \beta \frac{\| \outerIterate_{i-1} - \outerIterate_{i-2} \|}{\phi(\tau_{i-1})}$, i.e., a restart is triggered.
\end{lemma}

\begin{theorem}\label{thm:main-result}
Let $\epsilon \in (0, 1)$ and suppose that Assumptions~\ref{assume-Delta-r-marginal-gain}, \ref{assume:reduce-potential-function}, \ref{assume:contraction}, and~\ref{assume:error-bound} hold. Consider the sequence $\{ \outerIterate_i \}_{i=0}^{\infty}$, $\{ \tau_i \}_{i=1}^{\infty}$ generated by \callGenericAdaptiveRestartScheme{$(\outerIterate_0)$} for $\outerIterate_0 \in W$. Then for
$$
n = \bigg\lceil \log_{1/\beta}\left( \frac{1 + Q(\tau_1)}{\epsilon} \max\left\{ 1, \frac{ \phi( t^{*})}{\phi(\tau_1)} \right\} \right) \bigg\rceil
$$
the inequality $\frac{\| W^{*} - \outerIterate_{n} \|}{\| W^{*} - \outerIterate_0 \|} \le \epsilon$ holds and 
$\sum_{i=1}^{n} \tau_i \le t^{*} n + 2 (\tau_1 - t^{*})^{+}$.
\end{theorem}

The proof of Theorem~\ref{thm:main-result} appears in Appendix~\ref{app:proof-of-thm:main-result}. \emph{Proof sketch}. To show that $\frac{\| W^{*} - \outerIterate_{n} \|}{\| W^{*} - \outerIterate_0 \|} \le \epsilon$ we use that the potential function is decreasing by $\beta$ at each iteration and that it upper bounds the distance to optimality (Lemma~\ref{lemma:bound-r-star-new-by-r-old}). We  carefully use Lemma~\ref{lem:decrease-distance-potential-function-advanced} to bound $\sum_{i=1}^{n} \tau_i$. 

\section{Theory applied to specific algorithms}\label{sec:theory-to-specific}

\begin{definition}
A function $f : X \rightarrow \R$ is strongly convex if for all $x \in X$, $\lambda_{\min}(\grad^2 f(x)) \ge \alpha$.
\end{definition}

\begin{definition}
A differentiable function $f : W \rightarrow \R$ is $L$-smooth if $\| \grad f(w) - \grad f(w') \|_2 \le L \| w - w' \|_2$ for all $w, w' \in W$.
\end{definition}

Fact~\ref{fact:saddle-point-conversion} will be useful for showing standard sublinear bounds imply Assumption~\ref{assume:reduce-potential-function}. In particular, it allows us to change the center of the ball in a bound from $\bar{w}$ to $w$.

\begin{fact}\label{fact:saddle-point-conversion} 
For any $\beta \in (0,1)$, $\bar{w}, w \in W$ with $r = \| w - \bar{w} \|$ we have
$\sup_{(x,y) \in \ball{\beta r}{w}} f(w_x,y) -  f(x,w_y) \le 
\sup_{(x,y) \in \ball{(1 + \beta) r}{\bar{w}}} f(w_x,y) -  f(x,w_y)$.
\end{fact}

\begin{proof}
Holds because by the triangle inequality: $\ball{\beta r}{w}  \subseteq \ball{(1+\beta)r}{\bar{w}}$.
\end{proof}

\subsection{Primal dual hybrid gradient}\label{sec:pdhg}

\begin{assumption}[Standard setup for PDHG adapted from \cite{chambolle2011first}]\label{assume:standard-pdhg-setup}
Let $X$ be a closed convex set,
$f(x, y) = y^T A x + G(x) - F^{*}(y)$
where $F : X \rightarrow \R$ and $G : X \rightarrow \R$ are continuous convex functions. 
\end{assumption}
By Assumption~\ref{assume:standard-pdhg-setup}, the convex conjugate $F^{*} : Y \rightarrow \R$ is also a convex and continuous, with $Y$ closed and convex.
We will use the Euclidean norm to measure distances and $\PdhgStepSize \in (0,\infty)$ as the step size for the algorithm\footnote{For simplicity we assume that the primal and dual step sizes are equal. However, all these results go through when the primal and dual step sizes are different. One can obtain such results by diagonal scaling of the primal and dual variables, or by changing the norm to $\| w \| =  \sqrt{ \| w_x \|_2^2  /\PrimalStepSize + \| w_y \|_2^2 /\DualStepSize}$ where $\PrimalStepSize, \DualStepSize \in (0,\infty)$ are the primal and dual step sizes, respectively.}.

\begin{algorithm}[htpb]
\Fn{\InitializePDHG{$u$}}{
\Return $\mathbf{0}, u, u_x$\;
}
\Fn{\OneStepOfPDHG{$\bar{u}, u, \hat{x}, t, \PdhgStepSize$}}{
$u_y^{+} \gets \argmin_{y \in Y} F^{*}(y) - y^T A \hat{x} +
\frac{1}{2 \PdhgStepSize} \| y - u_y \|_2^2$ \;
$u_x^{+} \gets \argmin_{x \in X} G(x) + (u_y^{+})^T A x +
\frac{1}{2 \PdhgStepSize} \| x - u_x \|_2^2$ \;
$\hat{x}^{+} \gets 2 u_{x}^{+} - u_x$ \;
$\bar{u}^{+} \gets \frac{t-1}{t} \bar{u} + \frac{\bar{u}^{+}}{t}$\;
\Return $\bar{u}^{+}, u^{+}, \hat{x}^{+}$
}
 \Fn{\AdaptiveRestartPDHG{$\outerIterate_0, \PdhgStepSize$}}{
  \For{$i = 1, \dots, \infty$}{
 $w_{i}^{0}, u_{i}^{0}, \hat{x}_{i}^{0} \gets $
\InitializePDHG{$\outerIterate_{i-1}$} \;
$t \gets 0$\;
\Repeat{restart condition \eqref{eq:our-restart-condition} holds}{
$t \gets t + 1$\;
$w_{i}^{t}, u_{i}^t, \hat{x}_{i}^t \gets$ \OneStepOfPDHG{$w_{i-1}^t, u_{i-1}^{t}, \hat{x}_{i-1}^t, t, \PdhgStepSize$} \;
}
$\tau_i \gets t, \outerIterate_i \gets w_{i}^t$\;
}
}
\label{function:AdaptiveRestartPDHG}
\end{algorithm}

\begin{lemma}[Theorem~1 of \cite{chambolle2011first}]\label{lem:pdhg-basic-result}
Suppose Assumption~\ref{assume:standard-pdhg-setup} holds. Let $L = \| A \|_2$, then \callAdaptiveRestartPDHG{} with $L \PdhgStepSize < 1$ satisfies,
$\sup_{(\tilde{x},\tilde{y}) \in \ball{R}{\bar{y}}} f(x_{i}^t, \tilde{y}) - f(\tilde{x}, y_{i}^t) \le \frac{ R^2}{2\PdhgStepSize t}$
where $(x_{i}^t, y_{i}^t) := w_{i}^t$, $(\bar{x}, \bar{y}) := \outerIterate_{i-1}$,
for all $R \in \R^{+}$. Furthermore,
$\| w_{i}^{t} - w^{*} \|_2 \le (1 - \PdhgStepSize^2 L^2)^{-1/2} \| \outerIterate_{i-1} - w^{*} \|_2$
for all $w^{*} \in W^{*}$.
\end{lemma}

\begin{theorem}\label{main-pdhg-theorem}
Suppose Assumptions~\ref{assume:error-bound} and \ref{assume:standard-pdhg-setup} hold. Let $L = \| A \|_2$, $\theta \in (0,L]$, $\PdhgStepSize \in (0, 1/L)$, $\beta \in (0,1)$, and $\epsilon \in (0, 1)$. Define
$q := (1 - \PdhgStepSize^2 L^2)^{-1/2}$ and 
$t^{*} := \frac{ (1 + q)^2 (1 + \beta)^2}{\beta^2} \frac{1}{2 \gamma \theta} + 2$.
Consider the sequence $\{ \outerIterate_i \}_{i=0}^{\infty}$, $\{ \tau_i \}_{i=1}^{\infty}$ generated by \callAdaptiveRestartPDHG{} with $\outerIterate \in W$. Then for
$$
n = \bigg\lceil \log_{1/\beta}\left( \frac{1 + q}{\epsilon} \max\left\{ 1, \frac{ t^{*}}{\tau_1} \right\} \right) \bigg\rceil
$$
the inequality $\frac{\| W^{*} - \outerIterate_{n} \|}{\| W^{*} - \outerIterate_0 \|} \le \epsilon$ holds and 
$\sum_{i=1}^{n} \tau_i \le t^{*} n + 2 (\tau_1 - t^{*})^{+}$.
\end{theorem}

\begin{proof}
By Lemma~\ref{lem:Delta-r-marginal-gain} and Assumption~\ref{assume:standard-pdhg-setup}, Assumption~\ref{assume-Delta-r-marginal-gain} holds.
By combining Fact~\ref{fact:saddle-point-conversion} and Lemma~\ref{lem:pdhg-basic-result} with $R = (1 + \beta) r$, we observe that Assumption~\ref{assume:reduce-potential-function} holds with $\phi(t) = t$ and $C = \frac{1}{2 \gamma}$.  Lemma~\ref{lem:pdhg-basic-result} implies Assumption~\ref{assume:contraction} holds with $Q(t) := q$. Therefore we have established the premise of Theorem~\ref{thm:main-result} which implies the desired result.
\end{proof}

Corollary~\ref{coro:new-pdhg-bound} simply instantiates Theorem~\ref{main-pdhg-theorem} with $\beta=1/2$ and $\gamma=0.7 / L$.
These particular values are chosen to (approximately) minimize the constant factors.

\begin{corollary}\label{coro:new-pdhg-bound}
Suppose Assumption~\ref{assume:error-bound} and \ref{assume:standard-pdhg-setup} holds. Let $L = \| A \|_2$, $\theta \in (0,L]$, $\PdhgStepSize = 0.7 / L$, $\beta = 0.5$, $\epsilon \in (0, 1)$, and $\outerIterate_0 \in W$. 
 Then \callAdaptiveRestartPDHG{} requires at most 
$57 \frac{L}{\theta}  \ln\left( \frac{4}{ \epsilon} \right) + \max\left\{ 57 \frac{L}{\theta} \ln\left(\max\left\{ 1, \frac{77 L}{\theta \tau_1} \right\} \right),
 2 \tau_1 \right\}$
calls to \OneStepOfPDHG until some $\outerIterate_{n}$ satisfies $\frac{\| \outerIterate_{n} - w^{*} \|_2}{ \| \outerIterate_0 - w^{*} \|_2} \le \epsilon$.
\end{corollary}
\begin{proof}
Follows from the bound in Theorem~\ref{thm:new-agd-bound}. In particular, observing that for $\PdhgStepSize=0.7 / L$, $\beta=1/2$ we have $q = 0.51^{-1/2}$, $1 + 0.51^{-1/2} \le 3$, $t^{*} = \frac{ (1 + q)^2 (1 + \beta)^2}{\beta^2} \frac{1}{2 \gamma \theta} + 2 \le \frac{(1 + 0.51^{-1/2})^2 9}{0.7 \times 2} \frac{L}{\theta} + 2 \le 39 \frac{L}{\theta}$, $t^{*} / \ln(2) \le 57 \frac{L}{\theta}$. 
\end{proof}

Using almost the same argument, one can also show that our restart scheme applied to extragradient obtains the same worst-case complexity as Corollary~\ref{coro:new-pdhg-bound} (up to constant factors). We give the details for this in Appendix~\ref{sec:extragradient}.

Let us discuss our result in the bilinear setup. Recall that by Lemma~\ref{lem:minimum-nonsingular-value} in the bilinear setup $\theta$ is the minimum nonzero singular value of $A$ which we denote by $\sigma_{\min}$. Using standard lower bound arguments one can show that our bound $O(L / \sigma_{\min} \ln(1/\epsilon))$ is optimal for bilinear games (see Appendix~\ref{sec:lower-bounds}).
In contrast, the best known guarantees for saddle point algorithms in the bilinear setup \cite{mokhtari2019unified,zhangconvergence}\footnote{For comparing with \cite{mokhtari2019unified} note that in their setup $L = \sqrt{\lambda_{\max}(A^T A)}$ and $\sigma_{\min} = \sqrt{\lambda_{\min}(A^T A)}$.} study the last iterate and give bounds of the form $O(L^2 / \sigma_{\min}^2 \ln(1/\epsilon))$.

\subsection{Accelerated gradient descent}

For this subsection, let $f(x) = \SmoothFunction(x) + \ProxFunction(x)$ where $\SmoothFunction : \R^{n} \rightarrow \R$ is a smooth convex function and $\ProxFunction : \R^{n} \rightarrow \R$ is a continuous convex function which is possibly nonsmooth. We assume that the quantity $p_{\ell}(y) := \argmin_{x \in \R^{n}} \grad \SmoothFunction(y)^T (x - y) + \ProxFunction(x) + \frac{\ell}{2} \| x - y \|_2^2$ is easily computable. We use FISTA with backtracking line search \cite{beck2009fast} as the basic algorithm that we then integrate with \callGenericAdaptiveRestartScheme{}. We label this integrated algorithm \callAdaptiveRestartAGD{} and include its full description in Appendix~\ref{appendix:theory-to-specific:agd}. For conciseness, \textbf{define}: $\bar{\kappa} := \frac{L \backtrackingfactor}{\alpha}$
where $\backtrackingfactor \in (1,\infty)$ is the backtracking parameter for \callAdaptiveRestartAGD{}, $a$ is $L$-smooth and $\alpha$-strongly convex.
Theorem~\ref{thm:new-agd-bound} is an application of Theorem~\ref{thm:main-result} to AGD.
The proof involves establishing Assumptions~\ref{assume-Delta-r-marginal-gain}, \ref{assume:reduce-potential-function}, \ref{assume:contraction}, and~\ref{assume:error-bound} which is follows from standard results \cite{beck2009fast}. Corollary~\ref{coro:new-agd-bound} simply instantiates Theorem~\ref{thm:new-agd-bound} with $\beta=1/4$. Strong convexity is used both to establishes both Assumption~\ref{assume:reduce-potential-function} and \ref{assume:contraction}, and cannot be relaxed to quadratic growth.
The value $\beta=1/4$ is designed to (approximately) minimize the constant factors.
The proof of Theorem~\ref{thm:new-agd-bound} and Corollary~\ref{coro:new-agd-bound} appears in Appendix~\ref{sec:coro:thm:new-agd-bound}.

\begin{theorem}\label{thm:new-agd-bound}
Let $\beta, \epsilon \in (0, 1)$, $\outerIterate_0 \in W$, and suppose that $\SmoothFunction$ is $\alpha$-strongly convex and $L$-smooth with minimizer $w^{*}$. Define $t^{*} := 1 +  \sqrt{\hat{\kappa}} (\rho + \sqrt{\rho^2 + 4 \rho})$
where $\rho := \frac{1+\beta}{\beta}$. Consider the sequence $\{ \outerIterate_i \}_{i=0}^{\infty}$, $\{ \tau_i \}_{i=1}^{\infty}$ generated by \callAdaptiveRestartAGD{}. Then for
$$
n = \bigg\lceil \log_{1/\beta}( 2/\epsilon ) +  3 \log_{1/\beta} (\max\{ 1, t^{*} / \tau_1) \} ) \bigg\rceil
$$
the inequality $\frac{\| w^{*} - \outerIterate_{n} \|}{\| w^{*} - \outerIterate_0 \|} \le \epsilon$ holds and 
$\sum_{i=1}^{n} \tau_i \le t^{*} n + 2 (\tau_1 - t^{*})^{+}$.
\end{theorem}

\begin{corollary}\label{coro:new-agd-bound}
Suppose $\SmoothFunction$ is $\alpha$-strongly convex and $L$-smooth with minimizer $w^{*}$. Let $\beta = 1/4$, $\epsilon \in (0,1)$, and $\outerIterate_0 \in W$.
Then \callAdaptiveRestartAGD{} requires at most $8.5 \left( \sqrt{\bar{\kappa}} + 1 \right) \ln\left( \frac{8}{\epsilon} \right) +  \max\left\{
  26 \sqrt{\bar{\kappa}} \ln\left( \frac{12 \sqrt{\bar{\kappa}}}{\tau_1} \right),
 2 \tau_1 \right\}$ calls to \OneStepOfAGD until some $\outerIterate_{n}$ satisfies $\frac{\| w^{*} - \outerIterate_{n} \|_2}{\| w^{*} - \outerIterate_0 \|_2} \le \epsilon$.
\end{corollary}

Compared with other restart schemes in literature the coefficient of $8.5$ on the $\sqrt{\bar{\kappa}} \ln(1/\epsilon)$ term in Corollary~\ref{coro:new-agd-bound} is small. For \cite{alamo2019restart} (with $\mu_0 = 1$) the coefficient is $\approx 32$, for \cite{fercoq2019adaptive} it is $\approx 45$, and an exact hyperparameter search on the restart period yields a coefficient of $\approx 5.4$ (Remark~\ref{remark:optimal-scheme}).

\section{Numerical results}

We test the proposed adaptive restart scheme applied to the PDHG and AGD algorithms, comparing both with no restarts and with the best fixed-period restart chosen by grid search. For AGD we additionally compare with~\cite{o2015adaptive}. For PDHG, we modified the implementation in the Python library ODL~\cite{jonas_adler_2017_249479}, while we implemented AGD in Julia. All experiments use $\tau_1 = 1$. We note that in these experiments, the adaptive restart scheme is applied heuristically as we expect, but have not established, that Assumption~\ref{assume:error-bound} holds globally. In all the examples involving PDHG (Figure~\ref{fig:games} and \ref{fig:lp}) the average performs much worse than the current so we only plot the current iterate for the no-restart scheme.
For the restart schemes, as our theory suggests, we plot the running average since the last restart. All source code is available as part of the supplementary materials.
More details on the experiments are found in Appendix~\ref{app:more-experimental-details}. 

\subsection{Matrix games}\label{sec:matrix_games}

Given a matrix $A \in \R^{m \times n}$, a matrix game is the simple saddle-point problem defined by $f(x,y) = y^T Ax$, $X = \{ x \in \R^n : \sum_{i=1}^n x_i = 1, x \ge 0\}$, and $Y = \{ y \in \R^m : \sum_{i=1}^m x_i = 1, x \ge 0\}$. Following \cite{kroer2019}, we use synthetic instances from two families with $m = n = 100$: first where the coefficients of $A$ are sampled independently uniformly over $[-1, -\frac{1}{2}]$, and second from a standard Normal distribution. We measure the \textit{saddle-point residual} defined as $\Delta_{\infty}(w)$ at each iteration. We solve 50 instances from each family and plot the median residual with error bands for the 10th to 90th percentile. Results are in Figure~\ref{fig:games}. The adaptive scheme is consistent with the best fixed restart scheme, and both are dramatically better than PDHG without restarts.  

\begin{figure}[ht]
    \centering
    \includegraphics[scale=0.29]{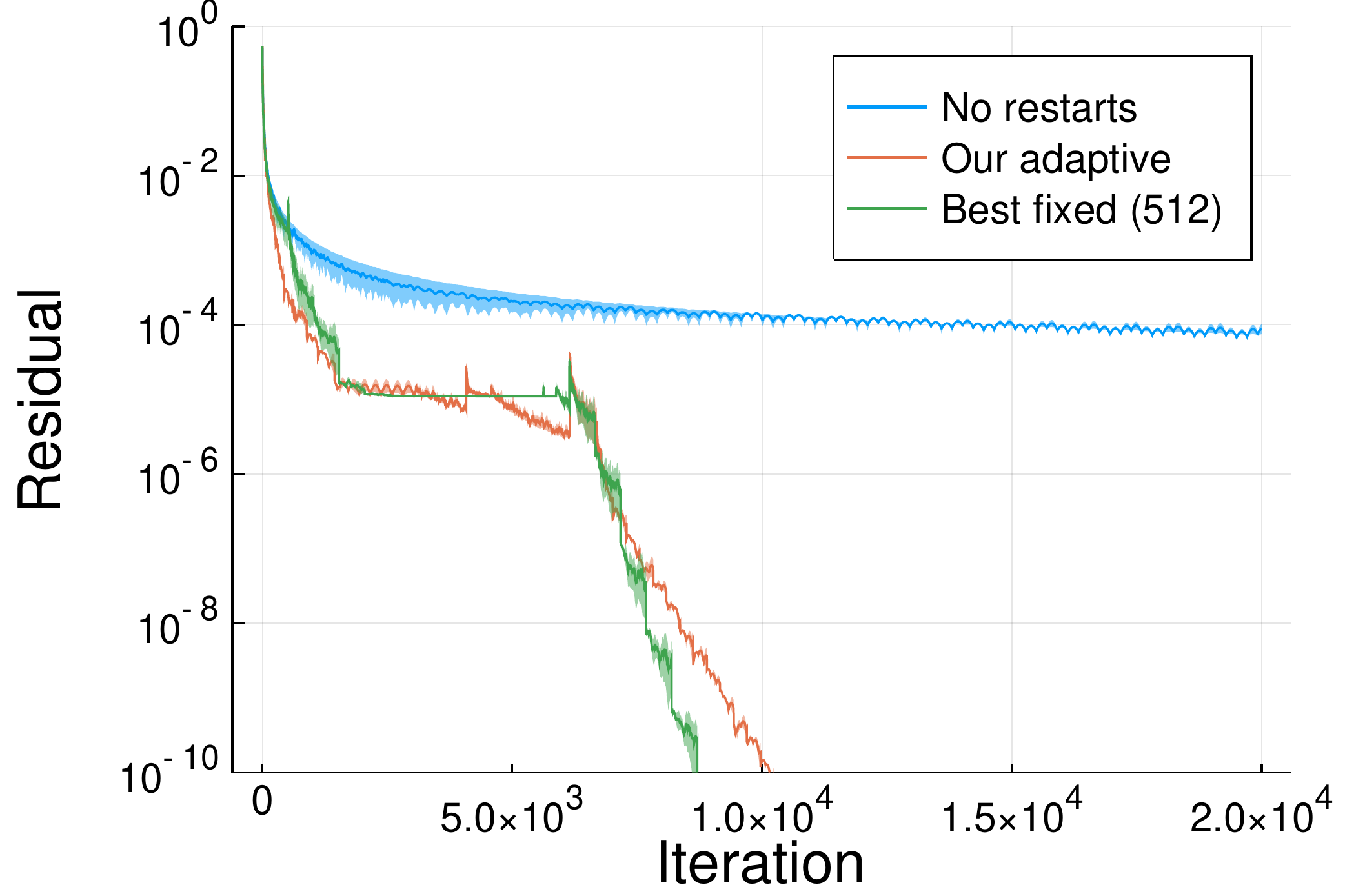}
    \includegraphics[scale=0.29]{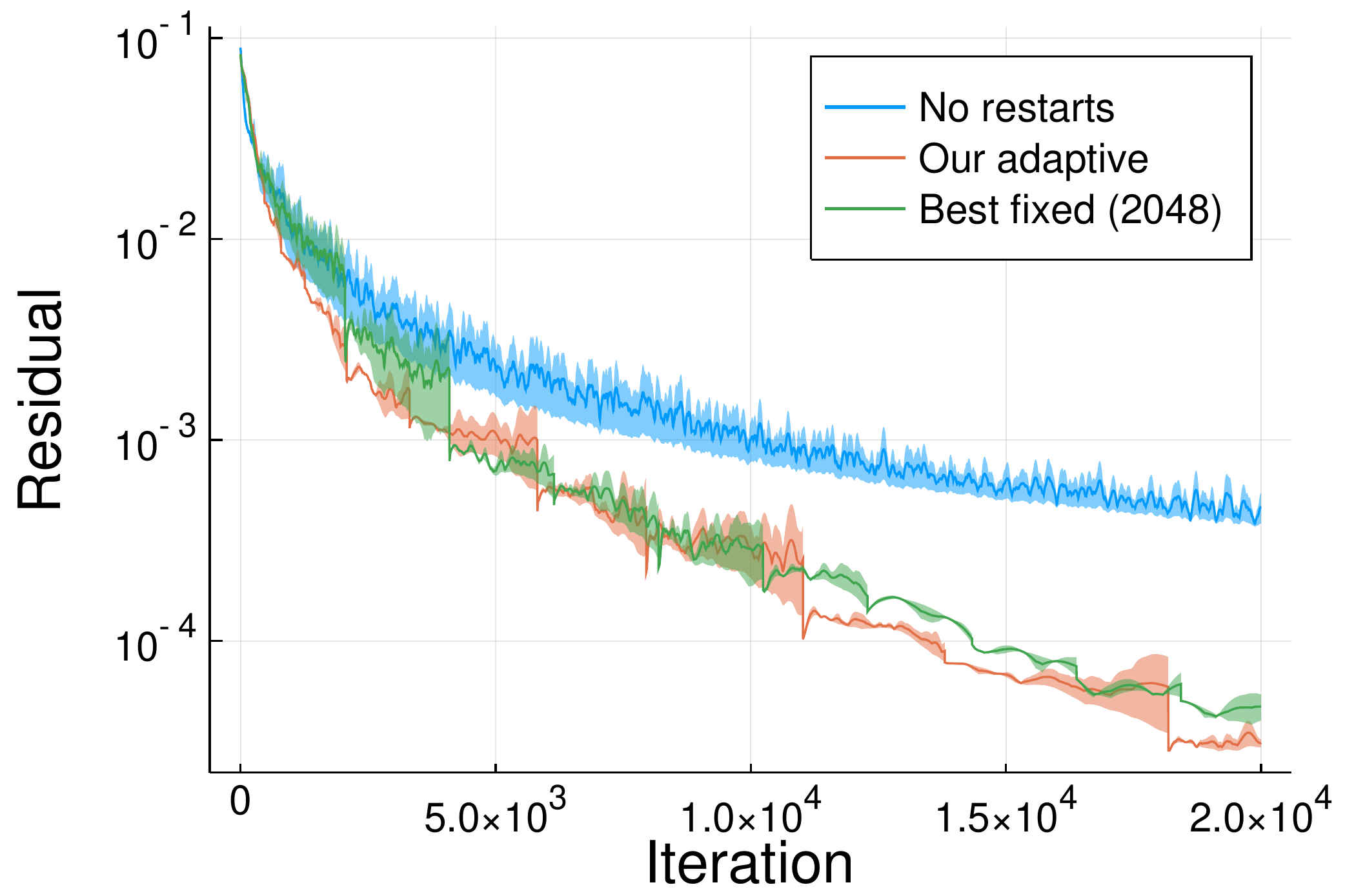}
    \caption{Matrix games (normal on left, uniform on right). The best fixed restart period is found via grid search on $\{8, 32, 128, 512, 2048\}$.
    }
    \label{fig:games}
\end{figure}

\subsection{Quadratic assignment problem relaxations}\label{sec:lp}

We select two linear programming (LP) instances, \texttt{qap15} and \texttt{nug08-3rd}, from the Mittelmann collection set~\cite{mittelmann_benchmark}, a standard benchmark set for LP. These two problems are relaxations of quadratic assignment problems~\cite{qapbounds2002}, a classical NP-hard combinatorial optimization problem.
These instances are known to be challenging for traditional solvers and amenable to first-order methods~\cite{GalabovaHall2020}.

We encode the LP $\min_x c^T x$ subject to $Ax = b, l \le x \le u$ in saddle-point form as $f(x,y) = c^T x - y^T Ax - b^T y$, $X = \{ x \in \R^n : l \le x \le u \}$, and $Y = \R^m$. Given a point $(\hat x, \hat y)$, the residual is measured as the $\ell_2$ norm of the vector concatenating the primal infeasibilities, dual infeasibilities, and the primal-dual objective gap. We additionally calculated the last iteration at which there is a change in the set of variables at either their upper or lower bounds as the \textit{last active set change}, a point that appears associated with the beginning of faster convergence. 
Results are in Figure~\ref{fig:lp}.

\begin{figure}[ht]
    \centering
    \includegraphics[scale=0.29]{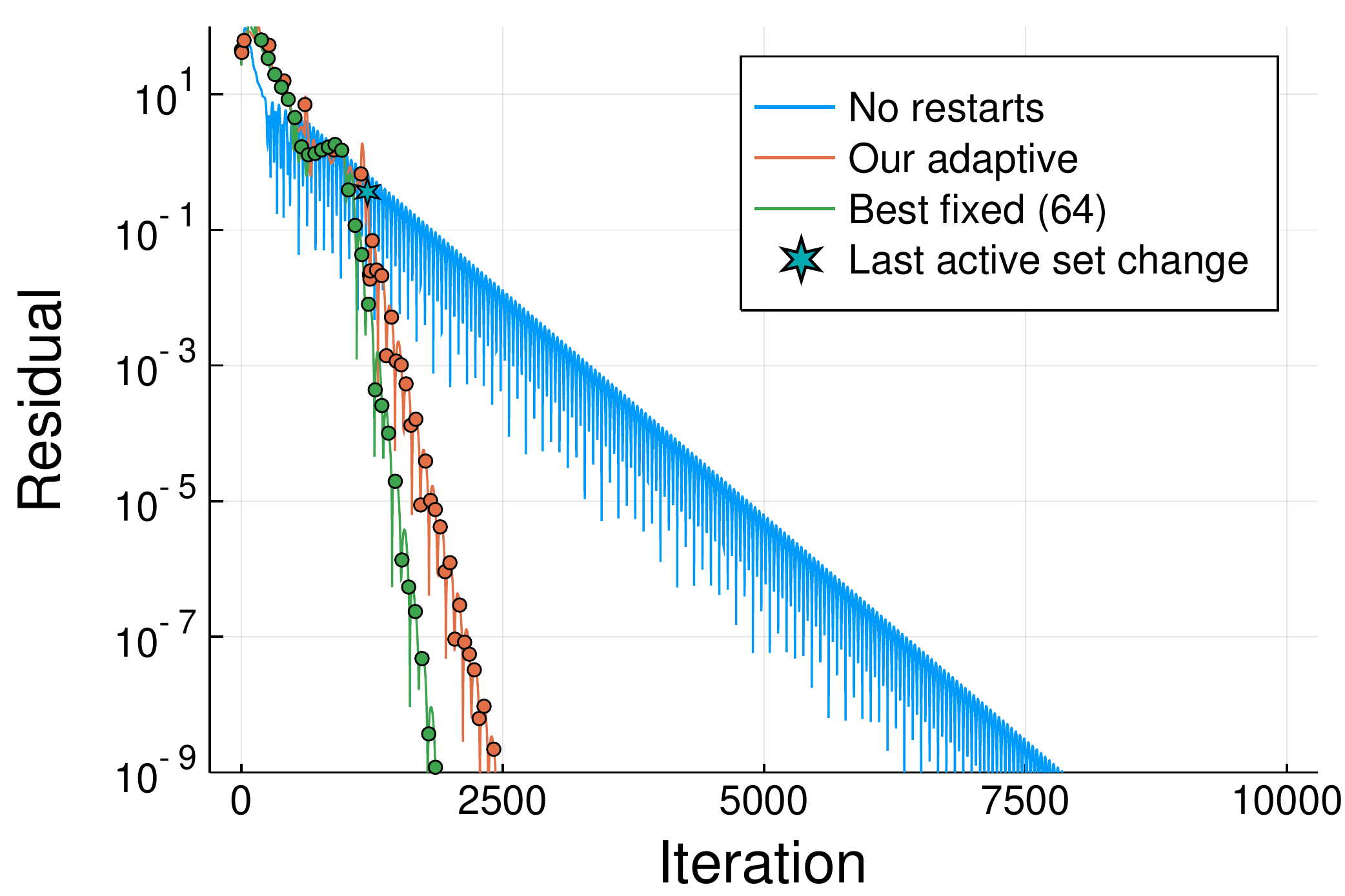}
    \includegraphics[scale=0.29]{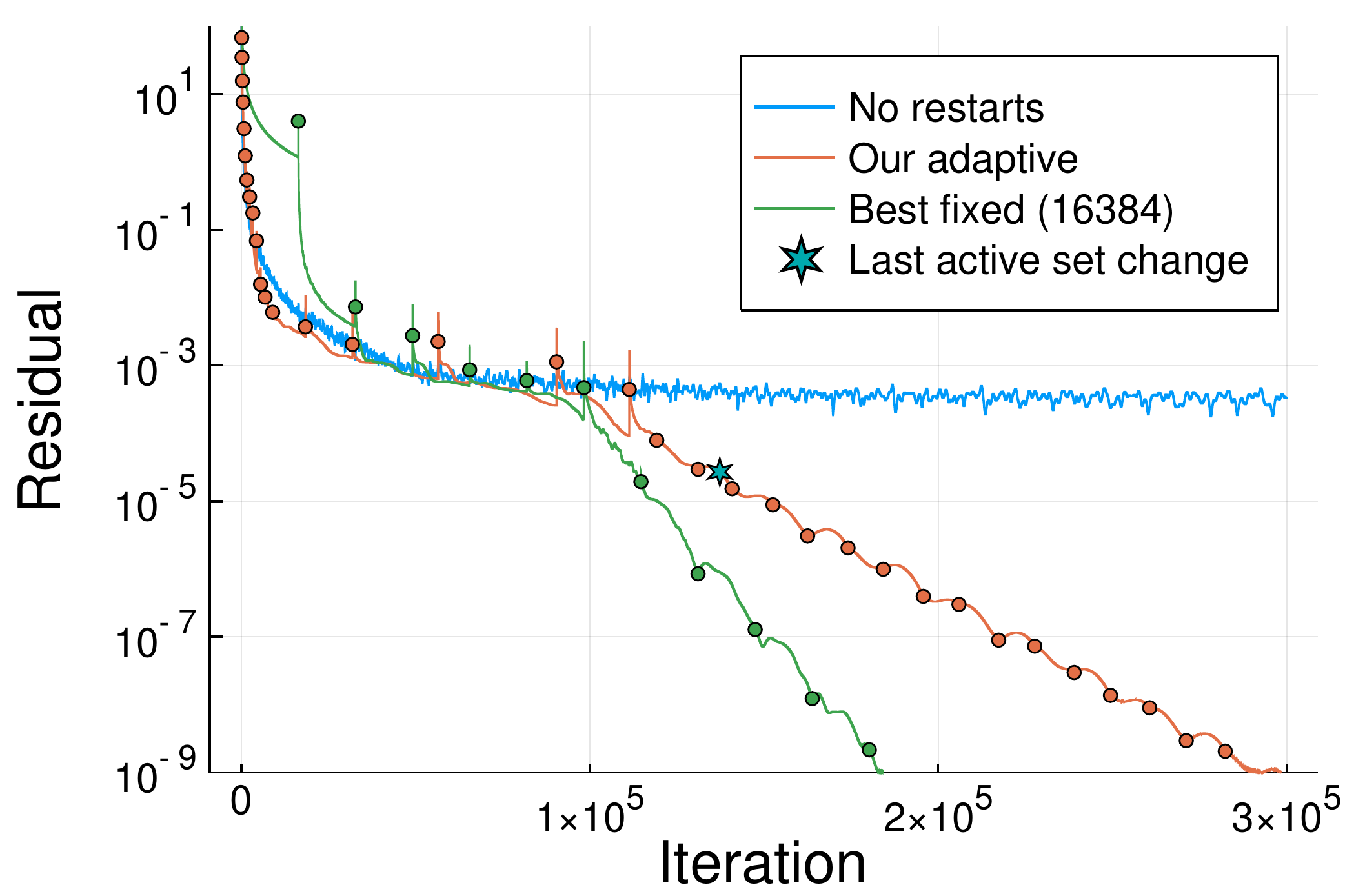}
    \caption{PDHG performance on LP instances \texttt{nug08-3rd} (left), \texttt{qap15} (right). The best fixed restart period is found via grid search on $\{64, 256, 1024, 4096, 16384, 65536\}$. Dots indicate restarts.
    }
    \label{fig:lp}
\end{figure}

\subsection{Logistic regression and LASSO}

To test our restart scheme for AGD we run on L1-regularizered logistic regression and LASSO problems downloaded from the LIBSVM dataset  \cite[\url{https://www.csie.ntu.edu.tw/~cjlin/libsvmtools/datasets/}]{chang2011libsvm}. We solve problems of the form $f(x) = \sum_{i=1}^n l_i(a_i^T x) + \lambda \| x \|_1$
where $\lambda$ is the regularization parameter, $a_i$ is the $i$th row of the data matrix, and $l_i$ is the loss function.
Let $b_i$ denote the data label. For LASSO we use
$l_i(c) = \frac{1}{2} (c - b_i)^2$, for L1-regularized logistic regression we use the log logistic loss $l_i(c) = \log(1 + \exp(c \sign(b_i)))$. The data matrix is preprocessed by (i) removing empty columns, (ii) adding an intercept, and (iii) normalizing the columns.
Statistics for the problems are given in Table~\ref{agd:dataset} in the Appendix.
We run \callAdaptiveRestartAGD{} with $\ell_0^0 = 1$, $\eta = 5/4$ and $\beta = 1/4$.
Figure~\ref{AGD:results} shows that our scheme is competitive with the function scheme of \citet{o2015adaptive} and in some instances (Duke breast cancer) does much better.

\begin{figure}[htbp]
    \centering
    \includegraphics[scale=0.21]{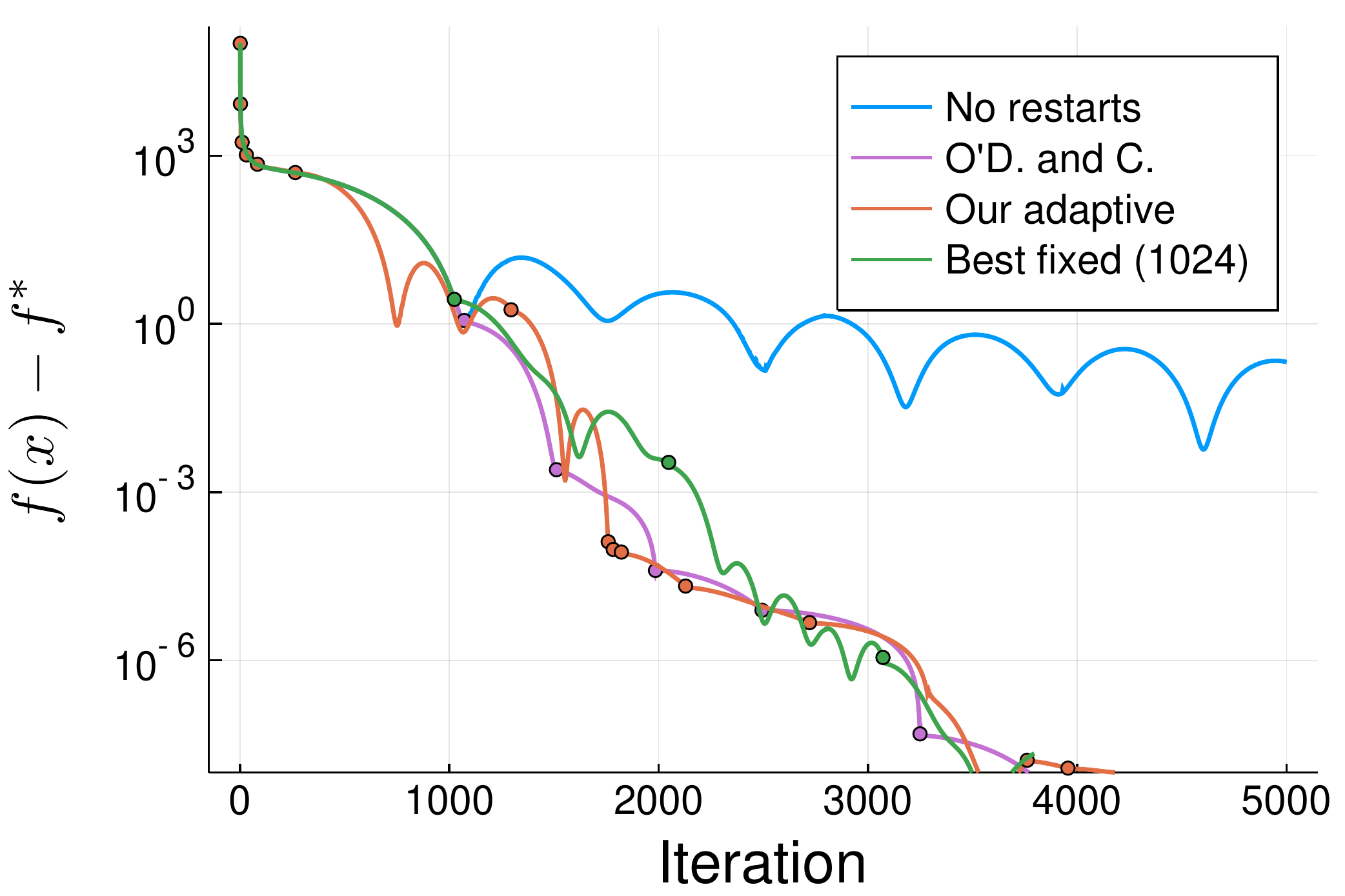} \includegraphics[scale=0.21]{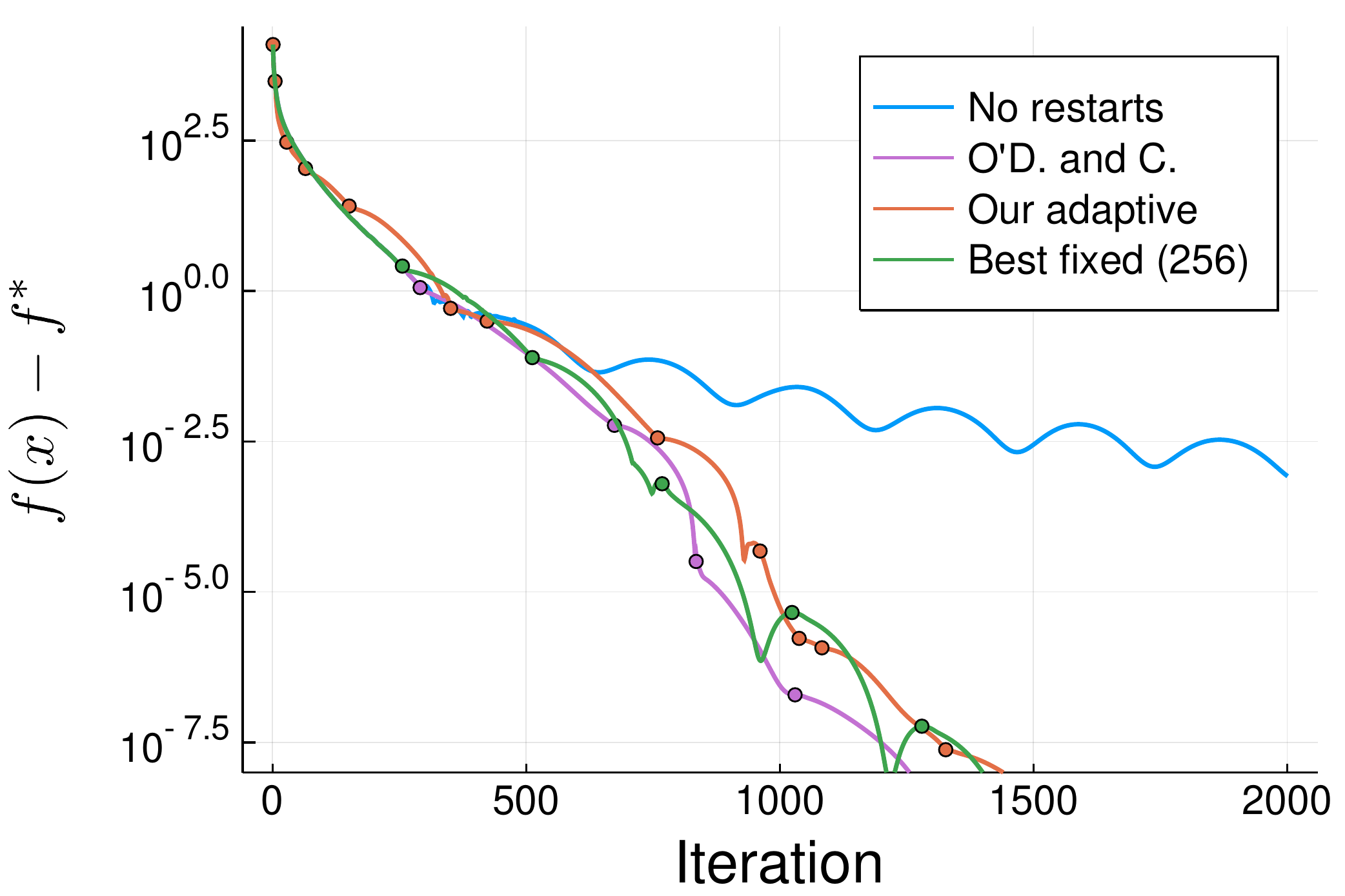}
    \includegraphics[scale=0.21]{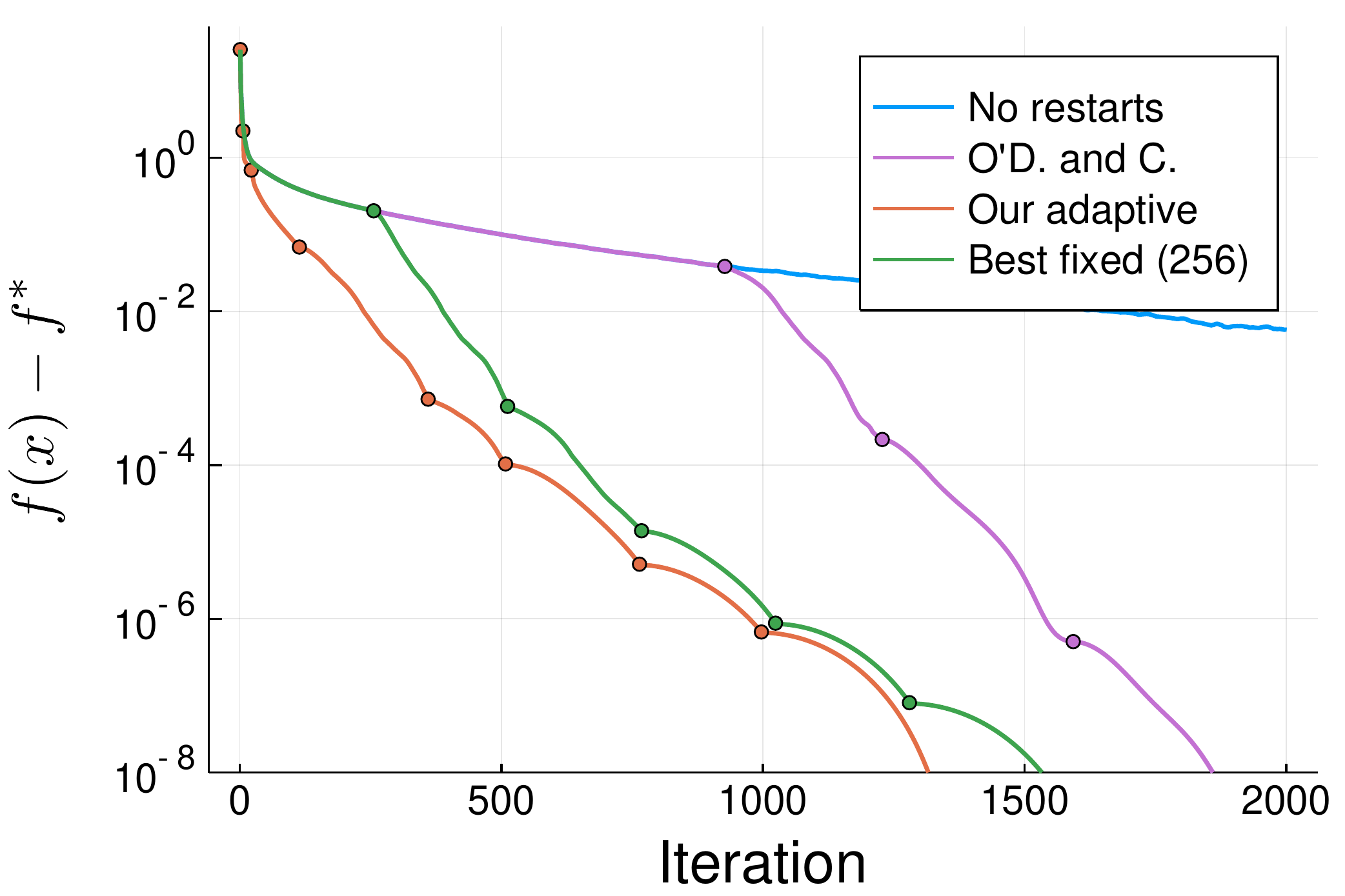}
    \caption{Tests on AGD. From left to right: E2006-tfidf (LASSO), rcv1.binary (logistic), Duke breast cancer (logistic). The best fixed restart period is found via grid search on $\{ 128, 256, 512, 1024, 2048 \}$. Dots indicate  restarts. ``O'D and C'' is the scheme of \cite{o2015adaptive}. }\label{AGD:results}
\end{figure}

\begin{ack}
We acknowledge David Applegate, Yair Carmon, Brendan O’Donoghue, Haihao Lu, and Warren Schudy for discussions on this work, and in particular: Haihao Lu for exchanges on the formulation of Assumption~\ref{assume:error-bound}, Warren Schudy for demonstrating the value of averaging in practice for saddle point algorithms, and David Applegate for comments on a draft of this paper.

\end{ack}

\bibliography{bio.bib}

\begin{thebibliography}{36}
\providecommand{\natexlab}[1]{#1}
\providecommand{\url}[1]{\texttt{#1}}
\expandafter\ifx\csname urlstyle\endcsname\relax
  \providecommand{\doi}[1]{doi: #1}\else
  \providecommand{\doi}{doi: \begingroup \urlstyle{rm}\Url}\fi

\bibitem[O'Donoghue and Candes(2015)]{o2015adaptive}
Brendan O'Donoghue and Emmanuel Candes.
\newblock Adaptive restart for accelerated gradient schemes.
\newblock \emph{Foundations of computational mathematics}, 15\penalty0
  (3):\penalty0 715--732, 2015.

\bibitem[Nemirovski and Nesterov(1985)]{nemirovskii1985optimal}
Arkadi~S Nemirovski and Yurii~E Nesterov.
\newblock Optimal methods of smooth convex minimization.
\newblock \emph{USSR Computational Mathematics and Mathematical Physics},
  25\penalty0 (2):\penalty0 21--30, 1985.

\bibitem[Roulet and d'Aspremont(2017)]{roulet2017sharpness}
Vincent Roulet and Alexandre d'Aspremont.
\newblock Sharpness, restart and acceleration.
\newblock In \emph{Advances in Neural Information Processing Systems}, pages
  1119--1129, 2017.

\bibitem[Nesterov(1983)]{nesterov1983method}
Yurii~E Nesterov.
\newblock A method for solving the convex programming problem with convergence
  rate $o (1/k^2)$.
\newblock In \emph{Soviet Mathematics Doklady}, volume~27, page 372–376,
  1983.

\bibitem[Lin and Xiao(2014)]{lin2014adaptive}
Qihang Lin and Lin Xiao.
\newblock An adaptive accelerated proximal gradient method and its homotopy
  continuation for sparse optimization.
\newblock In \emph{International Conference on Machine Learning}, pages 73--81,
  2014.

\bibitem[Nesterov(2013)]{nesterov2013gradient}
Yu~Nesterov.
\newblock Gradient methods for minimizing composite functions.
\newblock \emph{Mathematical Programming}, 140\penalty0 (1):\penalty0 125--161,
  2013.

\bibitem[Liu and Yang(2017)]{liu2017adaptive}
Mingrui Liu and Tianbao Yang.
\newblock Adaptive accelerated gradient converging method under {H}\"{o}lderian
  error bound condition.
\newblock In \emph{Advances in Neural Information Processing Systems}, pages
  3104--3114, 2017.

\bibitem[Alamo et~al.(2019)Alamo, Limon, and Krupa]{alamo2019restart}
Teodoro Alamo, Daniel Limon, and Pablo Krupa.
\newblock Restart {FISTA} with global linear convergence.
\newblock In \emph{2019 18th European Control Conference (ECC)}, pages
  1969--1974. IEEE, 2019.

\bibitem[Fercoq and Qu(2019)]{fercoq2019adaptive}
Olivier Fercoq and Zheng Qu.
\newblock Adaptive restart of accelerated gradient methods under local
  quadratic growth condition.
\newblock \emph{IMA Journal of Numerical Analysis}, 39\penalty0 (4):\penalty0
  2069--2095, 2019.

\bibitem[Chambolle and Pock(2011)]{chambolle2011first}
Antonin Chambolle and Thomas Pock.
\newblock A first-order primal-dual algorithm for convex problems with
  applications to imaging.
\newblock \emph{Journal of mathematical imaging and vision}, 40\penalty0
  (1):\penalty0 120--145, 2011.

\bibitem[Zhao(2020)]{zhao2020optimal}
Renbo Zhao.
\newblock Optimal stochastic algorithms for convex-concave saddle-point
  problems.
\newblock \emph{arXiv preprint arXiv:1903.01687}, 2020.

\bibitem[Tseng(1995)]{tseng1995linear}
Paul Tseng.
\newblock On linear convergence of iterative methods for the variational
  inequality problem.
\newblock \emph{Journal of Computational and Applied Mathematics}, 60\penalty0
  (1-2):\penalty0 237--252, 1995.

\bibitem[Hong and Luo(2017)]{hong2017linear}
Mingyi Hong and Zhi-Quan Luo.
\newblock On the linear convergence of the alternating direction method of
  multipliers.
\newblock \emph{Mathematical Programming}, 162\penalty0 (1-2):\penalty0
  165--199, 2017.

\bibitem[Eckstein(1989)]{eckstein1989splitting}
Jonathan Eckstein.
\newblock \emph{Splitting methods for monotone operators with applications to
  parallel optimization}.
\newblock PhD thesis, Massachusetts Institute of Technology, 1989.

\bibitem[Boley(2013)]{boley2013local}
Daniel Boley.
\newblock Local linear convergence of the alternating direction method of
  multipliers on quadratic or linear programs.
\newblock \emph{SIAM Journal on Optimization}, 23\penalty0 (4):\penalty0
  2183--2207, 2013.

\bibitem[Lu(2020)]{lu2020s}
Haihao Lu.
\newblock An $o(s^r)$-resolution ode framework for discrete-time optimization
  algorithms and applications to convex-concave saddle-point problems.
\newblock \emph{arXiv preprint arXiv:2001.08826}, 2020.

\bibitem[Du and Hu(2018)]{du2018linear}
Simon~S Du and Wei Hu.
\newblock Linear convergence of the primal-dual gradient method for
  convex-concave saddle point problems without strong convexity.
\newblock \emph{arXiv preprint arXiv:1802.01504}, 2018.

\bibitem[Mokhtari et~al.(2019)Mokhtari, Ozdaglar, and
  Pattathil]{mokhtari2019unified}
Aryan Mokhtari, Asuman Ozdaglar, and Sarath Pattathil.
\newblock A unified analysis of extra-gradient and optimistic gradient methods
  for saddle point problems: Proximal point approach.
\newblock \emph{arXiv preprint arXiv:1901.08511}, 2019.

\bibitem[Zhang and Yu(2020)]{zhangconvergence}
Guojun Zhang and Yaoliang Yu.
\newblock Convergence of gradient methods on bilinear zero-sum games.
\newblock \emph{ICLR}, 2020.

\bibitem[Hoffman(1952)]{hoffman1952approximate}
Alan~J Hoffman.
\newblock On approximate solutions of systems of linear inequalities.
\newblock \emph{Journal of Research of the National Bureau of Standards},
  49:\penalty0 263--265, 1952.

\bibitem[{\L}ojasiewicz(1959)]{lojasiewicz1961probleme}
Stanis{\l}aw {\L}ojasiewicz.
\newblock Sur le probleme de la division.
\newblock \emph{Studia Math.}, 18:\penalty0 87--136, 1959.

\bibitem[Drusvyatskiy and Lewis(2018)]{drusvyatskiy2018error}
Dmitriy Drusvyatskiy and Adrian~S Lewis.
\newblock Error bounds, quadratic growth, and linear convergence of proximal
  methods.
\newblock \emph{Mathematics of Operations Research}, 43\penalty0 (3):\penalty0
  919--948, 2018.

\bibitem[Fabian et~al.(2010)Fabian, Henrion, Kruger, and
  Outrata]{fabian2010error}
Marian~J Fabian, Ren{\'e} Henrion, Alexander~Y Kruger, and Ji{\v{r}}{\'\i}~V
  Outrata.
\newblock Error bounds: necessary and sufficient conditions.
\newblock \emph{Set-Valued and Variational Analysis}, 18\penalty0 (2):\penalty0
  121--149, 2010.

\bibitem[Beck and Teboulle(2009)]{beck2009fast}
Amir Beck and Marc Teboulle.
\newblock A fast iterative shrinkage-thresholding algorithm for linear inverse
  problems.
\newblock \emph{SIAM journal on imaging sciences}, 2\penalty0 (1):\penalty0
  183--202, 2009.

\bibitem[Adler et~al.(2017)Adler, Kohr, and Öktem]{jonas_adler_2017_249479}
Jonas Adler, Holger Kohr, and Ozan Öktem.
\newblock Operator discretization library (odl), January 2017.
\newblock URL \url{https://doi.org/10.5281/zenodo.249479}.

\bibitem[Gao et~al.(2019)Gao, Kroer, and Goldfarb]{kroer2019}
Yuan Gao, Christian Kroer, and Donald Goldfarb.
\newblock First-order methods with increasing iterate averaging for solving
  saddle-point problems.
\newblock \emph{CoRR}, abs/1903.10646, 2019.
\newblock URL \url{http://arxiv.org/abs/1903.10646}.

\bibitem[Mittelmann(2020)]{mittelmann_benchmark}
Hans~D. Mittelmann.
\newblock Benchmark of simplex {LP} solvers, Apr 2020.
\newblock URL \url{http://plato.asu.edu/ftp/lpsimp.html}.

\bibitem[Ramakrishnan et~al.(2002)Ramakrishnan, Resende, Ramachandran, and
  Pekny]{qapbounds2002}
K.~G. Ramakrishnan, M.~G.~C. Resende, B.~Ramachandran, and J.~F. Pekny.
\newblock \emph{Tight QAP bounds via linear programming}, pages 297--303.
\newblock World Scientific Publishing Co., 2002.
\newblock \doi{10.1142/9789812778215_0019}.
\newblock URL
  \url{https://www.worldscientific.com/doi/abs/10.1142/9789812778215_0019}.

\bibitem[Galabova and Hall(2020)]{GalabovaHall2020}
I.~L. Galabova and J.~A.~J. Hall.
\newblock The ‘idiot’ crash quadratic penalty algorithm for linear
  programming and its application to linearizations of quadratic assignment
  problems.
\newblock \emph{Optimization Methods and Software}, 35\penalty0 (3):\penalty0
  488--501, 2020.
\newblock \doi{10.1080/10556788.2019.1604702}.

\bibitem[Chang and Lin(2011)]{chang2011libsvm}
Chih-Chung Chang and Chih-Jen Lin.
\newblock {LIBSVM}: A library for support vector machines.
\newblock \emph{ACM transactions on intelligent systems and technology (TIST)},
  2\penalty0 (3):\penalty0 1--27, 2011.

\bibitem[Bubeck(2015)]{bubeck2015convex}
S{\'e}bastien Bubeck.
\newblock Convex optimization: Algorithms and complexity.
\newblock \emph{Foundations and Trends in Machine Learning}, 2015.

\bibitem[Nemirovski(2004)]{nemirovski2004prox}
Arkadi Nemirovski.
\newblock Prox-method with rate of convergence o (1/t) for variational
  inequalities with lipschitz continuous monotone operators and smooth
  convex-concave saddle point problems.
\newblock \emph{SIAM Journal on Optimization}, 15\penalty0 (1):\penalty0
  229--251, 2004.

\bibitem[Kogan et~al.(2009)Kogan, Levin, Routledge, Sagi, and
  Smith]{kogan2009predicting}
Shimon Kogan, Dimitry Levin, Bryan~R Routledge, Jacob~S Sagi, and Noah~A Smith.
\newblock Predicting risk from financial reports with regression.
\newblock In \emph{Proceedings of Human Language Technologies: The 2009 Annual
  Conference of the North American Chapter of the Association for Computational
  Linguistics}, pages 272--280, 2009.

\bibitem[Lewis et~al.(2004)Lewis, Yang, Rose, and Li]{lewis2004rcv1}
David~D Lewis, Yiming Yang, Tony~G Rose, and Fan Li.
\newblock Rcv1: A new benchmark collection for text categorization research.
\newblock \emph{Journal of machine learning research}, 5\penalty0
  (Apr):\penalty0 361--397, 2004.

\bibitem[West et~al.(2001)West, Blanchette, Dressman, Huang, Ishida, Spang,
  Zuzan, Olson, Marks, and Nevins]{west2001predicting}
Mike West, Carrie Blanchette, Holly Dressman, Erich Huang, Seiichi Ishida,
  Rainer Spang, Harry Zuzan, John~A Olson, Jeffrey~R Marks, and Joseph~R
  Nevins.
\newblock Predicting the clinical status of human breast cancer by using gene
  expression profiles.
\newblock \emph{Proceedings of the National Academy of Sciences}, 98\penalty0
  (20):\penalty0 11462--11467, 2001.

\bibitem[Shevade and Keerthi(2003)]{shevade2003simple}
Shirish~Krishnaj Shevade and S~Sathiya Keerthi.
\newblock A simple and efficient algorithm for gene selection using sparse
  logistic regression.
\newblock \emph{Bioinformatics}, 19\penalty0 (17):\penalty0 2246--2253, 2003.

\end{thebibliography}

\newpage

\appendix

\section{Proofs from Section~\ref{assume:required-for-adaptive}}

\subsection{Proof of Lemma~\ref{lem:Delta-r-marginal-gain}}\label{sec:proof-of:lem:Delta-r-marginal-gain}

\begin{proof}
Suppose that $r_a \ge r_b$ then it immediately follows that
$$
\sup_{(x,y) \in \ball{r_b}{w}} f(w_x, y) - f(x, w_y) \le \sup_{(x,y) \in \ball{r_a}{w}} f(w_x, y) - f(x, w_y).
$$
Therefore consider the case that $r_a \le r_{b}$.
Let $(x_a,y_a) \in \argmax_{(x,y) \in \ball{r_a}{w}} f(w_x,y) - f(x, w_y)$ and $(x_b,y_b) \in \argmax_{(x,y) \in \ball{r_b}{w}} f(w_x,y) - f(x, w_y)$.
By convexity $f(w_x + \lambda (x_b - w_x), w_y) \le (1 - \lambda) f(w) + \lambda f(x_b, w_y)$ with $\lambda = r_a / r_b$.
Rearranging this inequality yields:
\begin{flalign*}
f(w) - f(x_b,w_y) &\le \frac{ f(w) - f(w_x + \lambda (x_b - w_x), w_y)}{\lambda}   \\
&\le \frac{f(w) -  f(x_a, w_y)}{\lambda} \\
&= \frac{r_b}{r_a} (f(w) -  f(x_a, w_y))
\end{flalign*}
where the second inequality uses that $(w_x + \lambda (x_b - w_x), w_y) \in \ball{r_a}{x}$.
By the same argument (using concavity of $f$ in $y$ instead of convexity of $f$ in $x$),
$$
f(w_x, y_b) - f(w) \le \frac{r_b}{r_a} (f(w_x, y_a) -  f(w))
$$
Adding these inequalities together yields,
$$
f(w_x, y_b) - f(x_b, w_y) \le \frac{r_b}{r_a} (f(w_x, y_a) -  f(x_a, w_y))
$$
as required.
\end{proof}

\subsection{Proof of Lemma~\ref{lem:minimum-nonsingular-value}}\label{sec:proof-of:lem:minimum-nonsingular-value}

\begin{proof}
Start by considering the special case that $c = 0$, $b = 0$, and $A$ is diagonal (possibly nonsquare). Later we will show the general case reduces to this. Let $D$ be a diagonal matrix with entries $D_{ii} = \begin{cases} 1 & A_{ii} \neq 0 \\ 0 & A_{ii} = 0 \end{cases}$. Let $X^{*} \times Y^{*}$ denote the set of saddle points for this problem. Then
$X^{*} = \{ x : A x = 0 \} =  \{ x : D x = 0 \}$,  
$Y^{*} = \{ y : y^T A = 0 \} = \{ y : y^T D = 0 \}$.
In which case 
$$\| W^{*} - w \|_2 = \left\| \begin{pmatrix}
D w_x \\
D^T w_y
\end{pmatrix}
\right\|_2,$$
which implies that with $r^{*} = \| W^{*} - w \|_2$ that
\begin{flalign*}
\Delta_r(w) &= \sup_{(x,y) \in \ball{r^{*}}{w}} f(w_x, y) - f(x, w_y) \\
&= \sup_{(x,y) \in \ball{r^{*}}{w}} y^T A w_x-  x^T A^T w_y \\
&= \sup_{(x,y) \in \ball{r^{*}}{w}} (y - w_y)^T A w_x-  (x - w_x)^T A^T w_y \\
&= \sup_{(x,y) \in \ball{r^{*}}{\mathbf{0}}} y^T A w_x-  x^T A^T w_y \\
&= r^{*} \left\| \begin{pmatrix}
A w_x \\
A^T w_y
\end{pmatrix}
\right\|_2 \\
&\ge \sigmaMin r^{*}  \left\| \begin{pmatrix}
D w_x \\
D^T w_y
\end{pmatrix}
\right\|_2 \\
&= \sigmaMin (r^{*})^2.
\end{flalign*}
Next, consider the general case. We will reduce it to the case we just analyzed by shifting and rotating the space. In particular, consider the singular value decomposition $A = U \Sigma V^{T}$ where $\Sigma$ is an $n \times m$ diagonal matrix, and the matrices $U$ and $V$ are orthogonal. Define
$\bar{x} = V^T (x - x^{*})$ and $\bar{y} = U^T (y - y^{*})$. Therefore, 
\begin{flalign*}
&c^T x + y^T A x + b^T y \\
&=c^T x + (y - y^{*})^T A (x - x^{*}) + b^Ty + (y^{*})^T A x + y^T A x^{*} - (y^{*})^T A x^{*} \\
&= c^T x + (y - y^{*})^T A (x - x^{*}) + b^Ty - c^Tx - b^T y - (y^{*})^T A x^{*} \\
&= \bar{y}^T \Sigma \bar{x} - (y^{*})^T A x^{*}
\end{flalign*}
where the third inequality uses that $A x^{*} = -b$ and $(y^{*})^T A = -c$.
Note that because $U$ and $V$ are orthogonal: (i) $(\bar{y},\bar{x})$ is a saddle point for $\bar{y}^T \Sigma \bar{x}$ if and only if $(x,y)$ is a saddle point for $f$, (ii) $\| \bar{x} - \bar{x}' \|_2 = \| V^T (x - x') \|_2 =  \| x - x' \|_2$ and $\| \bar{y} - \bar{y}' \|_2 =  \| U^T (y - y') \|_2 = \| y - y' \|_2$ for $\bar{x}' = V^T (x - x^{*})$ and $\bar{y}' = U^T (y - y^*)$. It follows that the result holds in the general case.
\end{proof}

\section{Proof of results from Section~\ref{sec:proof-of-convergence}}

\subsection{Proof of Lemma~\ref{lem:decrease-distance-potential-function-advanced}}\label{app:proof-of:lem:decrease-distance-potential-function-advanced}

\begin{lemma}\label{lem:upper-bound-r-i-by-r-i-star}
Consider the sequence $\{w^t\}_{t=0}^{\infty}$ generated by \callGenericIterativeAlgorithm{$\outerIterate$}. 
Suppose that $t \in \N$ and Assumption~\ref{assume:contraction} holds then $\| w^{t} - \outerIterate \| \le (1 + Q(t)) \| W^{*} - \outerIterate \|$.
\end{lemma}
\begin{proof}
Let $w^{*} := \argmin_{w \in W^{*}} \| w - \outerIterate \|$. By the triangle inequality and definition of $Q(t)$ (as given in Assumption~\ref{assume:contraction}),
$\| w^{t} - \outerIterate \| \le \| w^{t} - w^{*} \| + \| w^{*} - \outerIterate \| \le (1 + Q(t)) \| \outerIterate - w^{*} \|$.
\end{proof}

\begin{lemma}\label{lemma:decrease-distance-potential-function}
Consider  \callGenericAdaptiveRestartScheme{}.
Suppose that Assumptions~\ref{assume-Delta-r-marginal-gain}, \ref{assume:reduce-potential-function}, \ref{assume:contraction}, and \ref{assume:error-bound} hold. If $t \in \N$ is such that $\phi(t) \ge (1+Q(t)) \frac{(1 + \beta)}{\beta} \max\left\{ \sqrt{ \hat{\kappa} \phi(\tau_{i-1}) },  \hat{\kappa} \frac{1 + \beta}{\beta} \right\}$ then
$\frac{\| w_i^t - \outerIterate_{i-1} \|}{\phi(t)} \le \beta \frac{\| \outerIterate_{i-1} - \outerIterate_{i-2} \|}{\phi(\tau_{i-1})}$.
\end{lemma}

\begin{proof}
If $\| W^{*} - \outerIterate_{i-1} \| \ge \beta \| \outerIterate_{i-1} - \outerIterate_{i-2} \|$ then
\begin{flalign*}
\frac{\| W^{*} - \outerIterate_{i-1} \|}{\phi(t)} &\le \frac{1}{\phi(t)} \frac{ (1 + \beta)^2 \hat{\kappa}}{\beta \phi(\tau_{i-1})} \| \outerIterate_{i-1} - \outerIterate_{i-2} \| & \text{(Lemma~\ref{lemma:bound-r-star-new-by-r-old})} \\
&\le \beta \frac{\| \outerIterate_{i-1} - \outerIterate_{i-2} \|}{\phi(\tau_{i-1})  (1+Q(t))} & \text{(By assumed the lower bound on $\phi(t)$)}.
\end{flalign*}
If $\| W^{*} - \outerIterate_{i-1} \| < \beta \| \outerIterate_{i-1} - \outerIterate_{i-2} \|$ then
\begin{flalign*}
\frac{\| W^{*} - \outerIterate_{i-1} \|}{\phi(t)} &\le \frac{1}{\phi(t) } (1 + \beta) \sqrt{  \frac{\hat{\kappa}}{\phi(\tau_{i-1})} }  \| \outerIterate_{i-1} - \outerIterate_{i-2} \| & \text{(Lemma~\ref{lemma:bound-r-star-new-by-r-old})} \\
&\le \beta \frac{\| \outerIterate_{i-1} - \outerIterate_{i-2} \|}{\phi(\tau_{i-1}) (1+Q(t))} & \text{(By assumed the lower bound on $\phi(t)$)}.
\end{flalign*}
Therefore,
$$
\frac{\| w_t^i - \outerIterate_{i-1} \|}{\phi(t)} \le \frac{\| W^{*} - \outerIterate_{i-1} \| (1+Q(t))}{\phi(t)} \le  \beta \frac{\| \outerIterate_{i-1} - \outerIterate_{i-2} \|}{\phi(\tau_{i-1})}
$$
where the first inequality uses Lemma~\ref{lem:upper-bound-r-i-by-r-i-star}, and the second inequality uses the established bound.
\end{proof}

\begin{proof}[Proof of Lemma~\ref{lem:decrease-distance-potential-function-advanced}]
The inequality $\phi(t) \ge \max\{ \sqrt{ \phi(t^{*}-2) \phi(\tau_{i-1}) }, \phi(t^{*}-2) \}$ implies that $\phi(t) \ge \phi(t^{*} - 2)$.
Therefore $t \ge t^{*} - 2 \Rightarrow Q(t) \le Q(t^{*}-2)$. It follows by \eqref{define:t-star} and $0 \le Q(t) \le Q(t^{*}-2)$ that
\begin{flalign*}  (1+Q(t)) \frac{1 + \beta}{\beta} \max\left\{ \sqrt{ \hat{\kappa} \phi(\tau_{i-1}) },  \hat{\kappa} \frac{1 + \beta}{\beta} \right\} \le \max\{ \sqrt{ \phi(t^{*}-2) \phi(\tau_{i-1}) }, \phi(t^{*}-2) \}.
\end{flalign*}
The result follows by Lemma~\ref{lemma:decrease-distance-potential-function}.
\end{proof}

\subsection{Proof of Theorem~\ref{thm:main-result}}\label{app:proof-of-thm:main-result}

\begin{lemma}\label{lem:bound-large-t}
Suppose that Assumptions~\ref{assume-Delta-r-marginal-gain}, \ref{assume:reduce-potential-function}, \ref{assume:contraction}, and~\ref{assume:error-bound} hold. Then, \callGenericAdaptiveRestartScheme{} generates a sequence $\{ \tau_i \}_{i=1}^{\infty}$ satisfying
$\tau_i \le \max\{ \tau_1, t^{*} \}$ and 
$\sum_{j=1}^{\infty} (\tau_j - t^{*})^{+} \le 2 (\tau_1 - t^{*})^{+}$.
\end{lemma}
\begin{proof}
By Lemma~\ref{lem:decrease-distance-potential-function-advanced} and \eqref{eq:our-restart-condition} we have
\begin{flalign*}
\phi(\tau_{i} - 1) < \max\left\{ \phi(t^{*}-2), \sqrt{ \phi(t^{*}-1) \phi(\tau_{i-1}) } \right\} \le  \max\left\{ \phi(t^{*} - 2), \phi\left(\frac{\tau_{i-1} + t^{*}}{2} - 1 \right) \right\}.
\end{flalign*}
If $\tau_{i-1} \le t^{*} - 2$ then
$\phi(\tau_{i} - 1) < \phi(t^{*} - 2) \Rightarrow \tau_i - 1 < t^{*} - 2 \Rightarrow \tau_i < t^{*} - 1$.
If $\tau_{i-1} > t^{*} - 2$ then we have
$\phi(\tau_i - 1) < \phi\left(\frac{\tau_{i-1} + t^{*}}{2} - 1 \right)$ which by monotonicity of $\phi$ implies 
\begin{flalign}
\label{eq:tau-i-inductive-upper-bound}
\tau_i \le \frac{\tau_{i-1} + t^{*}}{2} 
\end{flalign}
If $\tau_{i-1} \le t^{*}$ then \eqref{eq:tau-i-inductive-upper-bound} implies $\tau_i \le t^{*}$.
Subtracting $t^{*}$ from \eqref{eq:tau-i-inductive-upper-bound} and using induction it follows that if $\tau_j > t^{*}$ for all $j \in \{1, \dots, i\}$ then
$\tau_j - t^{*} \le \frac{\tau_{j-1} - t^{*}}{2} 
\le \frac{\tau_{1} - t^{*}}{2^{j-1}}$. This  implies for all $n \in \N$ that
$\sum_{i=1}^{n}(\tau_i - t^{*})^{+} \le (\tau_1 - t^{*})^{+} \sum_{i=1}^{n} 2^{1-i} \le 2 (\tau_1 - t^{*})^{+}$,
where the last inequality uses the standard bound on the sum of a geometric series.
\end{proof}

\begin{lemma}\label{lem:t-hat-bound}
Consider the sequence $\{ \outerIterate_i \}_{i=0}^{\infty}$ generated by \callGenericAdaptiveRestartScheme{}.
Suppose that Assumptions~\ref{assume-Delta-r-marginal-gain}, \ref{assume:reduce-potential-function}, \ref{assume:contraction}, and~\ref{assume:error-bound} hold.
Then for all $n \in \N$,
$\| W^{*} - \outerIterate_{n} \| \le \| \outerIterate_1 - \outerIterate_{0} \|  \max\left\{ 1, \frac{ \phi( t^{*})}{\phi(\tau_1)} \right\} \beta^{n}$.
\end{lemma}

\begin{proof}
By definition of \callGenericAdaptiveRestartScheme{}, for all $i \in \N$ with $i > 1$,
$\frac{\| \outerIterate_{i} - \outerIterate_{i-1} \|}{\phi(\tau_i)} \le \beta \frac{\| \outerIterate_{i-1} - \outerIterate_{i-2} \|}{\phi(\tau_{i-1})}$ which by induction implies 
\begin{flalign}\label{eq:omega-difference-bound}
\| \outerIterate_{n} - \outerIterate_{n-1} \| \le \| \outerIterate_{1} - \outerIterate_{0} \|  \frac{\phi(\tau_{n})}{\phi(\tau_{1})} \beta^{n-1}.
\end{flalign}
Moreover, in the case that $\| \outerIterate_{n} - \outerIterate_{n-1} \| \ge \beta \| \outerIterate_{n} - \outerIterate_{n-1} \|$ we have
\begin{flalign*}
\| W^{*} - \outerIterate_{n} \| &\le \frac{(1 + \beta)^2}{\beta} \frac{\hat{\kappa}  }{\phi(\tau_{n})} \| \outerIterate_{n} - \outerIterate_{n-1} \| & \text{by Lemma~\ref{lemma:bound-r-star-new-by-r-old},} \\
& \le \beta  \frac{\phi(t^{*})  }{\phi(\tau_{n})} \| \outerIterate_{n} - \outerIterate_{n-1} \| & \text{by \eqref{define:t-star}.} \\
\end{flalign*}
Therefore, in the general case
\begin{flalign*}
\| W^{*} - \outerIterate_{n} \| &\le \beta  \max\left\{ 1, \frac{\phi(t^{*})  }{\phi(\tau_{n})}  \right\} \| \outerIterate_{n} - \outerIterate_{n-1} \|.
\end{flalign*}

By the previous inequality and \eqref{eq:omega-difference-bound} we deduce
$$
\| W^{*} - \outerIterate_{n} \| \le \| \outerIterate_{1} - \outerIterate_{0} \|  \beta \max\left\{ 1, \frac{\phi(t^{*})  }{\phi(\tau_{n})}  \right\}  \frac{\phi(\tau_{n})}{\phi(\tau_{1})} \beta^{n} = \| \outerIterate_{1} - \outerIterate_{0} \| \max\left\{ \frac{\phi(\tau_{n})}{\phi(\tau_{1})}, \frac{\phi(t^{*})  }{\phi(\tau_{1})}  \right\} \beta^{n}.
$$
Lemma~\ref{lem:bound-large-t} implies that either $\tau_{n} \le \tau_1$ or $\tau_{n} \le t^{*}$. If 
$\tau_{n} \le \tau_1$ then by monotonicity $\frac{\phi(\tau_{n})}{\phi(\tau_{1})} \le 1$.
If $\tau_{n} \le t^{*}$ then by monotonicity $\phi(\tau_{n}) \le \phi(t^{*})$. In both cases, $\max\left\{ \frac{\phi(\tau_{n})}{\phi(\tau_{1})}, \frac{\phi(t^{*})  }{\phi(\tau_{1})}  \right\} \le \max\left\{ 1, \frac{\phi(t^{*})  }{\phi(\tau_{1})}  \right\}$.
\end{proof}

\begin{proof}[Proof of Theorem~\ref{thm:main-result}]
We have
\begin{flalign*}
\| W^{*} - \outerIterate_{n} \| &\le \| \outerIterate_{1} - \outerIterate_{0} \| \max\left\{ 1, \frac{ \phi( t^{*})}{\phi(\tau_1)} \right\} \beta^{n} & \text{by Lemma~\ref{lem:t-hat-bound}} \\
&\le  \| \outerIterate_{1} - \outerIterate_{0} \| \frac{\epsilon}{1 + Q(\tau_1)}  & \text{by definition of $n$} \\
&\le \epsilon \| W^{*} - \outerIterate_{0} \| & \text{by Lemma~\ref{lem:upper-bound-r-i-by-r-i-star}}
\end{flalign*}
If $\tau_1 \le t^{*}$ then by Lemma~\ref{lem:t-hat-bound} we deduce that $\tau_i \le t^{*}$ for all $i$. If $\tau_1 > t^{*}$ then using Lemma~\ref{lem:bound-large-t} we deduce that
$\sum_{i=1}^{n} \tau_i \le
t^{*} n + \sum_{i=1}^{n} (\tau_i - t^{*})^{+}
\le t^{*} n + 2 (\tau_1 - t^{*})$.
\end{proof}

\section{Supplementary material for Section~\ref{sec:theory-to-specific}}\label{appendix:theory-to-specific}

\subsection{Accelerated gradient descent}\label{appendix:theory-to-specific:agd}

\begin{lemma}\label{lem:std-agd-result}
Let $f$ be convex and $L$-smooth with $L \ge \ell_{0}^{0}$.
At any inner iteration of \callAdaptiveRestartAGD{} and for all $x \in X$ the following inequality holds,
$$
f(w_{i}^{t}) - f(x) \le \frac{2 \backtrackingfactor L \| \outerIterate_{i-1} - x \|_2^2}{(t+1)^2}.
$$
\end{lemma}

\begin{proof}
See Theorem~4.4 of \cite{beck2009fast} and note that the proof does not use that $x^{*}$ is a minimizer, so we can replace it with any $x \in X$.
\end{proof}

\begin{algorithm}[htpb]
\label{function:AdaptiveRestartAGD}
\Fn{\InitializeAGD{$w$}}{
\Return $w, w, 1$\;
}
\Fn{\OneStepOfAGD{$w, v, \lambda, \ell, \backtrackingfactor$}}{
Find the smallest nonnegative integer $k$ such that with $\bar{\ell} = \backtrackingfactor^{k} \ell$
 $$
 f(p_{\bar{\ell}}(v)) \le \SmoothFunction(v) + (p_{\bar{\ell}}(v) - v)^T \grad \SmoothFunction(v) + \frac{\bar{\ell}}{2} \| p_{\bar{\ell}}(v) - v \|_2^2 + \ProxFunction(p_{\bar{\ell}}(v)).
 $$
Set $\ell_{+} = \backtrackingfactor^{k} \ell$ and
 compute 
 $w_{+} \gets p_{\ell_{+}}(v)$\;
 $\lambda_{+} \gets \frac{1 + \sqrt{1 + 4 \lambda^2}}{2}$\;
  $v_{+} = w_{+} +  \frac{\lambda-1}{\lambda_{+}} (w_{+} - w)$\;
 \Return $w_{+}, v_{+}, \lambda_{+}, \ell_{+}$
}
 \Fn{\AdaptiveRestartAGD{$\outerIterate_0, \ell_0^0, \eta$}}{
  \For{$i = 1, \dots, \infty$}{
 $w_{i}^0, v_{i}^0, \lambda_{i}^0 \gets $
\InitializeAGD{$\outerIterate_{i-1}, \ell_{i-1}^{0}$} \;
$t \gets 0$ \;
\Repeat{restart condition \eqref{eq:our-restart-condition} holds}{
$t \gets t + 1$\;
$w_{i}^t, v_{i}^t, \lambda_{i}^t, \ell_{i}^{t} \gets$ \OneStepOfAGD{$w_i^{t-1}, v_i^{t-1},  \lambda_i^{t-1}, \ell_{i}^{t-1}, \backtrackingfactor$} \;
}
$\tau_i \gets t$, $\outerIterate_i \gets w_i^t$\;
}
}
\end{algorithm}

\subsubsection{Proof of Theorem~\ref{thm:new-agd-bound} and Corollary~\ref{coro:new-agd-bound}}\label{sec:coro:thm:new-agd-bound}

\begin{proof}[Proof of Theorem~\ref{thm:new-agd-bound}]
Assumption~\ref{assume-Delta-r-marginal-gain} holds by Lemma~\ref{lem:Delta-r-marginal-gain}. 
By Fact~\ref{fact:saddle-point-conversion} and \ref{lem:std-agd-result}, Assumption~\ref{assume:reduce-potential-function} holds with $C = 2L \backtrackingfactor$ and $\phi(k) = (k+1)^{2}$. By Lemma~\ref{lem:std-agd-result} with $x = w^{*}$ and strong convexity,
$$
\frac{2 \backtrackingfactor L\| \outerIterate_{i-1} - w^{*} \|_2^2}{(t+1)^2} \ge f(w_{i}^t) - f(w^{*}) \ge \frac{\alpha}{2} \| w_{i}^{t} - w^{*} \|_2^2 \Rightarrow 2 \sqrt{\frac{\backtrackingfactor L}{\alpha (t + 1)^2}} \| \outerIterate_{i-1} - w^{*} \|_2  \ge \| w_{i}^{t} - w^{*} \|_2.
$$
Therefore, Assumption~\ref{assume:contraction} holds with $Q(t) = \frac{2}{t + 1} \sqrt{\bar{\kappa}}$.
Assumption~\ref{assume:error-bound} with $\theta=\alpha/2$ holds by $\alpha$-strong convexity. 

With the premise of Theorem~\ref{thm:main-result} established it only remains to make the value of $t^{*}$ explicit and simplify the bound. 
From the definition of $C$ and $\theta$,
$\hat{\kappa} = 4 \bar{\kappa}$.
We claim that
$$
t^{*} = 1 +  \sqrt{\hat{\kappa}} (\rho + \sqrt{\rho^2 + 4 \rho})
$$
is a solution to \eqref{define:t-star}. To establish this claim note that this value of $t^{*}$ implies that
\begin{flalign*}
& (t^{*} - 1)^2 - 2 \rho \sqrt{\hat{\kappa}} ( t^{*} - 1 ) - 4 \rho \hat{\kappa}  = 0 & \text{by the quadratic formula} \\
&\Rightarrow t^{*}-1 = 2 \rho \sqrt{\hat{\kappa}} \left( 1+ \frac{2 \sqrt{\hat{\kappa}}}{t^{*}-1} \right) & \text{moving terms to the RHS and dividing by $t^{*}-1$} \\
&\Rightarrow \frac{(t^{*} - 1)^2}{\left(1 + \frac{2 \sqrt{\bar{\kappa}}}{t^{*}-1}\right)^2} = 4 \rho^2 \hat{\kappa} & \text{squaring both sides and rearranging,}
\end{flalign*}
which establishes \eqref{define:t-star}.

As $\beta \in (0,1) \Rightarrow \rho \ge 2$ the previous equation implies $t^{*} \ge 2 \sqrt{\hat{\kappa}}$. Therefore
\begin{flalign*}
\left( 1 + \frac{2\sqrt{\bar{\kappa}} }{\tau_1 + 1} \right) \max\left\{ 1, \frac{ (t^{*} + 1)^2}{(\tau_1 + 1)^2} \right\} \le 2 \max\left\{ 1, \frac{(t^{*}+1)^3}{(\tau_1+1)^3} \right\}.
\end{flalign*}
\end{proof}

\begin{proof}[Proof of Corollary~\ref{coro:new-agd-bound}]
Follows by substituting $\beta=1/4$ into the bound in Theorem~\ref{thm:new-agd-bound}, which yields
$$
t^{*} = 1 +  (5 + 3 \sqrt{5}) \sqrt{\hat{\kappa}}
$$
$$
\frac{t^{*}}{\ln(4)} \le 1 + 8.5 \sqrt{\hat{\kappa}}.
$$
\end{proof}

\begin{remark}\label{remark:optimal-scheme}
It is worth contrasting the bound in Corollary~\ref{coro:new-agd-bound} with the bound that one achieves if $\bar{\kappa}$ is known and one picks the fixed restart period $T$ that minimizes the worst-case bound. By Lemma~\ref{lem:std-agd-result} and strong convexity, $\| w_{i+1} - w^{*} \|_2^2 \le \frac{4 \bar{\kappa}}{(T+1)^2}  \| w_i - w^{*} \|_2^2 \Rightarrow \frac{\| w_{n+1} - w^{*} \|_2^2}{ \| w_1 - w^{*} \|_2^2 } \le \left( \frac{4 \bar{\kappa}}{(T+1)^2} \right)^{n} \Rightarrow n \le \ln\left(\frac{\| w_1 - w^{*} \|_2}{\| w_{n+1} - w^{*} \|_2 } \right) / \ln\left( \frac{(T+1)^2}{4 \bar{\kappa}} \right)$, which implies that $\| w_{n+1} - w^{*} \|_2 \le \epsilon$ after at most
$T \ln\left(\frac{\| w_1 - w^{*} \|_2}{\e \epsilon} \right) / \ln\left( \frac{(T+1)^2}{4 \bar{\kappa}} \right)$ iterations. We can (approximately) minimize this upper bound with respect to $T$ by exactly minimizing $T \ln(\frac{T}{2 \sqrt{\bar{\kappa}}})$ yielding $T = 2 \e \sqrt{\bar{\kappa}}$ and an overall bound of $( 2 e \sqrt{\bar{\kappa}} + 1) \ln\left( \frac{\| w_1 - w^{*} \|}{\e \epsilon} \right)$.
\end{remark}

\subsection{Extragradient}\label{sec:extragradient}

This subsection is based on \citet{bubeck2015convex}. 
For general Bregman divergences this algorithm is called mirror prox; when the Euclidean norm is used (as is the case here) it is known as extragradient.

Define,
$$
g(w) := \begin{pmatrix}
\grad_x f(x,y) \\
-\grad_y f(x,y)
\end{pmatrix}.
$$

\begin{algorithm}[htpb]
\Fn{\InitializeExtragradient{$u$}}{
\Return $\mathbf{0}, u$\;
}
\Fn{\OneStepOfExtragradient{$\bar{u}, u, t, \PdhgStepSize$}}{
$v \gets \argmin_{w \in W} g(u)^T w + \frac{1}{\PdhgStepSize} \| w - u \|_2^2$ \;
$u^{+} \gets \argmin_{w \in W} g(v)^T w + \frac{1}{\PdhgStepSize} \| w - u \|_2^2$ \;
$\bar{u}^{+} \gets \frac{t-1}{t} \bar{u} + \frac{\bar{u}^{+}}{t}$\;
\Return $\bar{u}^{+}, u^{+}$
}
 \Fn{\AdaptiveRestartExtragradient{$\outerIterate_0, \PdhgStepSize$}}{
  \For{$i = 1, \dots, \infty$}{
 $w_{i}^{0}, u_{i}^{0} \gets $
\InitializeExtragradient{$\outerIterate_{i-1}$} \;
$t \gets 0$\;
\Repeat{restart condition \eqref{eq:our-restart-condition} holds}{
$t \gets t + 1$\;
$w_{i}^{t}, u_{i}^t \gets$ \OneStepOfExtragradient{$w_{i-1}^t, u_{i-1}^{t}, t, \PdhgStepSize$} \;
}
$\tau_i \gets t, \outerIterate_i \gets w_{i}^t$\;
}
}
\label{function:AdaptiveRestartExtragradient}
\end{algorithm}

Lemma~\ref{lem:extragradient-basic-result} follows the proof of \citet[Theorem~4.4]{bubeck2015convex}.
The original proof is given by \citet[Proposition 2.2.]{nemirovski2004prox}.

\begin{lemma}\label{lem:extragradient-basic-result}
Suppose $f(x,y)$ is $L$-smooth, convex in $x$ for all $y \in Y$, and concave in $y$ for all $x \in X$. 
Then \callAdaptiveRestartExtragradient{} with $\PdhgStepSize \in (0,1/L]$ satisfies,
$\sup_{(\tilde{x},\tilde{y}) \in \ball{R}{\bar{w}}} f(x_{i}^t, \tilde{y}) - f(\tilde{x}, y_{i}^t) \le \frac{R^2}{\PdhgStepSize t}$
where $(x_{i}^t, y_{i}^t) := w_{i}^t$, $(\bar{x}, \bar{y}) := \outerIterate_{i-1}$,
for all $R \in \R^{+}$. Furthermore,
$\| w_{i}^{t} - w^{*} \|_2 \le \| \outerIterate_{i-1} - w^{*} \|_2$
for all $w^{*} \in W^{*}$.
\end{lemma}

\begin{proof}
Repeating the proof of
Theorem~4.4 of \cite{bubeck2015convex} with $\grad f(y_{t+1})$ replaced with $g(u_i^{t})$ yields
$$
g(u_i^t)^T ( u_{i}^t - w ) \le \frac{\| w - u_{i}^{t-1} \|_2^2 - \| w - u_{i}^{t} \|_2^2}{\gamma}. %
$$
From the previous inequality, using that $f$ is convex-concave we deduce
\begin{flalign}\label{eq:bound-f-gap}
f(\hat{x}_{i}^t, y) - f(x, \hat{y}_{i}^t) \le \frac{\| w - u_{i}^{t-1} \|_2^2 - \| w - u_{i}^{t} \|_2^2}{\gamma}
\end{flalign}
with $w = (x,y)$ and $u_i^t = (\hat{x}_i^t, \hat{y}_i^t)$.
By Jensen's inequality and telescoping, \eqref{eq:bound-f-gap} implies
$$
f(x_{i}^t, y) - f(x, y_{i}^t) \le \frac{1}{t} \sum_{k=1}^{t} (f(\hat{x}_{i}^k, y) - f(x, \hat{y}_{i}^k))
\le \frac{\| w_i^t - \outerIterate_{i-1} \|_2^2}{\gamma t}.
$$
Therefore the bound on the duality gap holds. Also, since for any saddle point $(x^{*}, y^{*}) = w^{*} \in W^{*}$ we have $f(x_{i}^t, y^{*}) - f(x^{*}, y_{i}^t) \ge 0$ by \eqref{eq:bound-f-gap} we deduce $\| w^{*} - u_{i}^{t} \|_2 \le \| w^{*} - u_{i}^{t-1} \|_2$. This implies $\| w^{*} - u_{i}^{t} \|_2 \le \| w^{*} - \outerIterate_{i-1} \|_2$ and therefore by the triangle inequality,
$$
\| w^{*} - w_{i}^{t} \|_2 \le \frac{1}{t} \sum_{k=1}^{t} \| w^{*} - u_{i}^{k} \|_2 \le \| w^{*} - \outerIterate_{i-1} \|_2.
$$
\end{proof}

\begin{theorem}\label{main-extragradient-theorem}
Suppose $f(x,y)$ is $L$-smooth, convex in $x$ for all $y \in Y$, and concave in $y$ for all $x \in X$. Further suppose Assumptions~\ref{assume:error-bound} holds. Let $\theta \in (0,L]$, $\PdhgStepSize \in (0, 1/L]$, $\beta \in (0,1)$, and $\epsilon \in (0, 1)$. Define
$t^{*} := \frac{ 4 (1 + \beta)^2}{\beta^2} \frac{1}{\gamma \theta} + 2$.
Consider the sequence $\{ \outerIterate_i \}_{i=0}^{\infty}$, $\{ \tau_i \}_{i=1}^{\infty}$ generated by \callAdaptiveRestartExtragradient{}. Then for
$$
n = \bigg\lceil \log_{1/\beta}\left( \frac{2}{\epsilon} \max\left\{ 1, \frac{ t^{*}}{\tau_1} \right\} \right) \bigg\rceil
$$
the inequality $\frac{\| W^{*} - \outerIterate_{n} \|}{\| W^{*} - \outerIterate_0 \|} \le \epsilon$ holds and 
$\sum_{i=1}^{n} \tau_i \le t^{*} n + 2 (\tau_1 - t^{*})^{+}$.
\end{theorem}

\begin{proof}
By Lemma~\ref{lem:Delta-r-marginal-gain}, Assumption~\ref{assume-Delta-r-marginal-gain} holds.
By combining Fact~\ref{fact:saddle-point-conversion} and~\ref{lem:extragradient-basic-result} with $R = (1 + \beta) r$, we observe that Assumption~\ref{assume:reduce-potential-function} holds with $\phi(t) = t$ and $C = 2 / \gamma$.  Lemma~\ref{lem:extragradient-basic-result} implies Assumption~\ref{assume:contraction} holds with $Q(t) := 1$. Therefore we have established the premise of Theorem~\ref{thm:main-result} which implies the desired result.
\end{proof}

\subsection{Lower bounds for saddle point algorithms}\label{sec:lower-bounds}

With appropriate indexing of their iterates, saddle point algorithms such as primal-dual hybrid gradient, extragradient and their restarted variants satisfy
\begin{subequations}
\label{generic-saddle-point-algorithm}
\begin{flalign}
x_{t} &\in x_0 + \text{span}(\grad_x f(x_0, y_0), \dots, \grad_x f(x_{t-1}, y_{t-1})) \\
y_{t} &\in y_0 + \text{span}(\grad_y f(x_0, y_0), \dots, \grad_y f(x_{t}, y_{t})).
\end{flalign}
\end{subequations}
On bilinear games with $c = y_0 = x_0 = \mathbf{0}$, \eqref{generic-saddle-point-algorithm} simplifies to
\begin{flalign*}
y_{t} &\in \text{span}(A x_{0} - b, \dots, A x_{t-1} - b) \\
x_{t} &\in \text{span}(A^T y_{0}, \dots, A^T y_{t}),
\end{flalign*}
which implies
\begin{flalign}\label{span:xt}
x_{t} &\in \text{span}(A^T (A x_{0} - b), \dots, A^T (A x_{t-1} - b)).
\end{flalign}
Therefore to form a lower bound for saddle point algorithms it will suffice to consider algorithms that satisfy \eqref{span:xt}. We provide a lower bound for algorithms satisfying \eqref{span:xt} in Theorem~\ref{theorem:lb-saddle-point}. The proof of Theorem~\ref{theorem:lb-saddle-point} is essentially identical to the proof of \citet[Theorem~3.15]{bubeck2015convex}, the proof is just reframed in terms of a bilinear game instead of minimizing a quadratic.

\begin{theorem}\label{theorem:lb-saddle-point}
For any $0 < \gamma_{\min} < \gamma_{\max}$,
there exists a matrix $A$ and vector $b$ with $\gamma_{\max}$ and minimum singular value greater than $\gamma_{\min}$ such that for any black-box procedure satisfying \eqref{span:xt} with $x_0 = \mathbf{0}$, one has
$$
\| x_t - x^{*} \| \ge \left( \frac{\frac{\gamma_{\max}}{\gamma_{\min}} - 1}{\frac{\gamma_{\max}}{\gamma_{\min}} + 1} \right)^{t - 1} \| x_0 - x^{*} \|
$$
where $x^{*}$ is the unique primal solution to the saddle point problem $f(x,y) = y^T A x + b^T y$.
\end{theorem}

\begin{proof}
Let $\mathbf{T} \in \R^{k \times k}$ be a tridiagonal matrix with $2$ on the diagonal and $-1$ on the upper and lower diagonals. Note that $x^T \mathbf{T} x = 2 \sum_{i=1}^{k} x(i)^2 - 2 \sum_{i=1}^{k-1} x(i) x(i+1) = x(1)^2 + x(k)^2 + \sum_{i=1}^{k-1} (x(i) - x(i+1))^2$, which implies $\mathbf{0} \preceq	 \mathbf{T} \preceq	 4 \mathbf{I}$. Define,
$A := \sqrt{\frac{\gamma_{\max}^2 - \gamma_{\min}^2}{4} \mathbf{T} + \gamma_{\min}^2 \mathbf{I}}, \quad b := 2 (A^T)^{-1} e_1$.
Since
$A^T A = A^2 = \frac{\gamma_{\max}^2 - \gamma_{\min}^2}{4} \mathbf{T} + \gamma_{\min}^2 \mathbf{I}$,
we deduce that the minimum singular value of $A$ is greater than $\sqrt{\lambda_{\min}(A^T A)} \ge \gamma_{\min}$, and the maximum singular value of $A$ is less than $\sqrt{\lambda_{\max}(A^T A)} \le \gamma_{\max}$. The remainder of the proof continues exactly as the proof of \citet[Theorem~3.15]{bubeck2015convex} except with $\kappa = \gamma_{\max}^2 / \gamma_{\min}^2$ and $\alpha = \gamma_{\min}^2$. 
\end{proof}

Note, for $\gamma_{\max} / \gamma_{\min} \gg 1$ we have
$$
\left( \frac{\frac{\gamma_{\max}}{\gamma_{\min}} - 1}{\frac{\gamma_{\max}}{\gamma_{\min}} + 1} \right)^{t - 1} \approx \exp\left( -2 (t - 1) \frac{ \gamma_{\max}}{\gamma_{\min}} \right).
$$

\section{More experimental details}\label{app:more-experimental-details}

\subsection{Matrix games}

The step size used for the matrix games is $\PdhgStepSize = \frac{\sqrt{0.9}}{||A||_2}$ where the operator norm $||A||_2$ is computed using \texttt{numpy.linalg.norm}. The constant $\sqrt{0.9}$ was taken from ODL and was not tuned. $\beta = \frac{1}{2}$ is used in the restart scheme.

\subsection{Quadratic assignment problem relaxations}
For each instance we first fix the ratio of primal and dual step sizes by a hyperparameter sweep over $\{ 10^{-5}, 10^{-4}, 10^{-3}, 10^{-2}, 10^{-1}, 10^{0}, 10^{1}, 10^{2}, 10^{3}, 10^{4}, 10^{5}\}$. We run PDHG without restarts for 1000 iterations and select the ratio $r$ such that the final iterate has the smallest residual. Given $r$, the dual step size is chosen as $\PdhgStepSize_y = \sqrt{\frac{0.9}{r}} \frac{1}{||A||_2}$ and  the primal step size chosen as $\PdhgStepSize_x = r\PdhgStepSize_y$. The operator norm $||A||_2$ is estimated using ODL's implementation of power iteration. $\beta = \frac{1}{2}$ is used in the restart scheme. Table~\ref{qap:dataset} lists the dimensions of the two instances.

\begin{table}[htbp]
    \centering
    \begin{tabular}{c|c|c|c}
      Instance & \# variables & \# constraints & \# nonzeros in $A$ \\
      \hline
     \texttt{nug08-3rd} & 20,448 & 19,728 & 139,008 \\
     \texttt{qap15} & 22,275 & 6,331 & 110,700 \\
    \end{tabular}
    \caption{Dimensions of the quadratic assignment problem relaxations.}
    \label{qap:dataset}
\end{table}

\subsection{Logistic regression and LASSO}

Given runs with different fixed periods, we pick the algorithm that gets the function value below $10^{-8}$ earliest (in terms of number of iterations) as the best. Table~\ref{agd:dataset} lists statistics of the instances.

\begin{table}[htbp]
    \centering
    \begin{tabular}{c|c|c|c|c}
      Dataset  & Problem & \# data points & \# features & regularizer \\
      \hline
     E2006-tfidf \cite{kogan2009predicting} & LASSO & 16,087 & 150,360 & 1.0 \\
     rcv1.binary \cite{lewis2004rcv1} & regularized logistic loss & 20,242 & 47,236 & $10^{-1}$ \\
     Duke breast cancer \cite{west2001predicting,shevade2003simple} & regularized logistic loss & 44 & 7,129 & $10^{-2}$  \\
    \end{tabular}
    \caption{Statistics of test problems for AGD.}
    \label{agd:dataset}
\end{table}

\subsection{A hard example for the function scheme of O'Donoghue and Candes \cite{o2015adaptive}}
Consider the following $1$-smooth function:
$$
h_{\delta}(\eta) = \begin{cases}
\eta^2 / 2 & \eta \ge -\delta \\
-\delta \eta - \delta^2 / 2 & \eta < -\delta.
\end{cases}
$$
which is plotted in Figure~\ref{figure:plot-h} for $\delta = 0.1$.
Using this function we will construct a hard example for AGD,
\begin{flalign}\label{eq:hard-example}
f(x) = \sum_{i=1}^{n} i h_{\delta}(x_i) + \frac{\alpha}{2} \| x \|_2^2
\end{flalign}
with $n = 500$, $\delta = 10^{-4}$, $\alpha = 10^{-4}$ starting from $x = -\mathbf{1}$.
Note this problem is $n$-smooth and $\alpha$-strongly convex. The unique minimizer is $x = \mathbf{0}$.
In this situation the restart interval that minimizes the worst-case bound is $\approx \e \sqrt{n / \alpha} \approx 6000$. From Figure~\ref{figure:hard-example}, we can see that our method and hyper-parameter searching on the restart period produces restarts of roughly this magnitude. 
The function restart scheme of O'Donoghue and Candes \cite{o2015adaptive} in contrast restarts too frequently, causing the algorithm to run slower than vanilla AGD.
Intuitively, the sharp transition in the smoothness at $-\delta$ causes the function value to regularly increase and stopping the scheme of O'Donoghue and Candes \cite{o2015adaptive} from building up momentum. However, once it enters the neighborhood of the minimizer where the function is quadratic it quickly converges.

\begin{figure}[ht]
    \centering
   \includegraphics[scale=0.3]{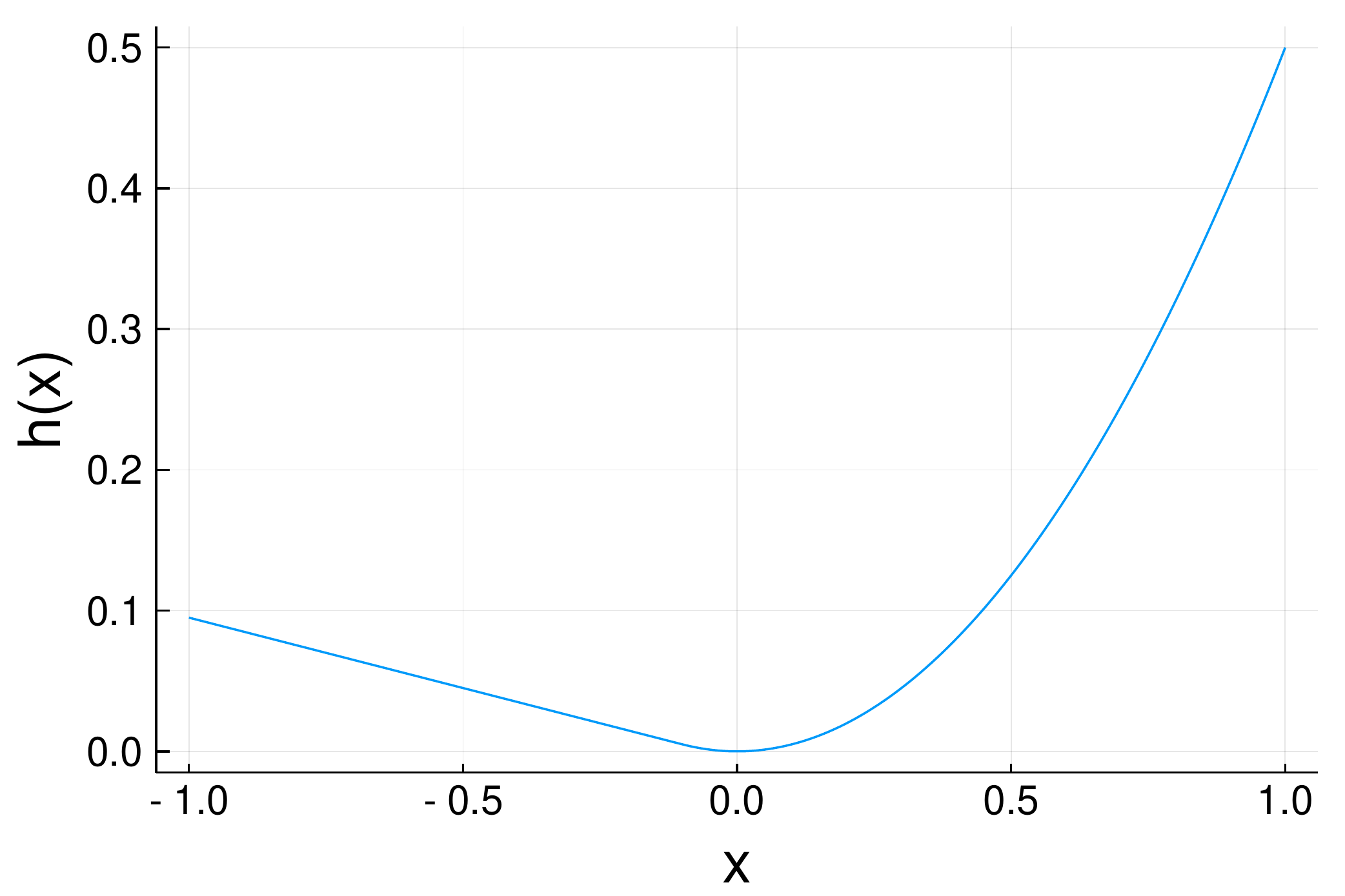}
    \caption{Plot of the function $h_{\delta}$ for $\delta=0.1$}
    \label{figure:plot-h}
\end{figure}

\begin{figure}[ht]
\centering 
\includegraphics[scale=0.3]{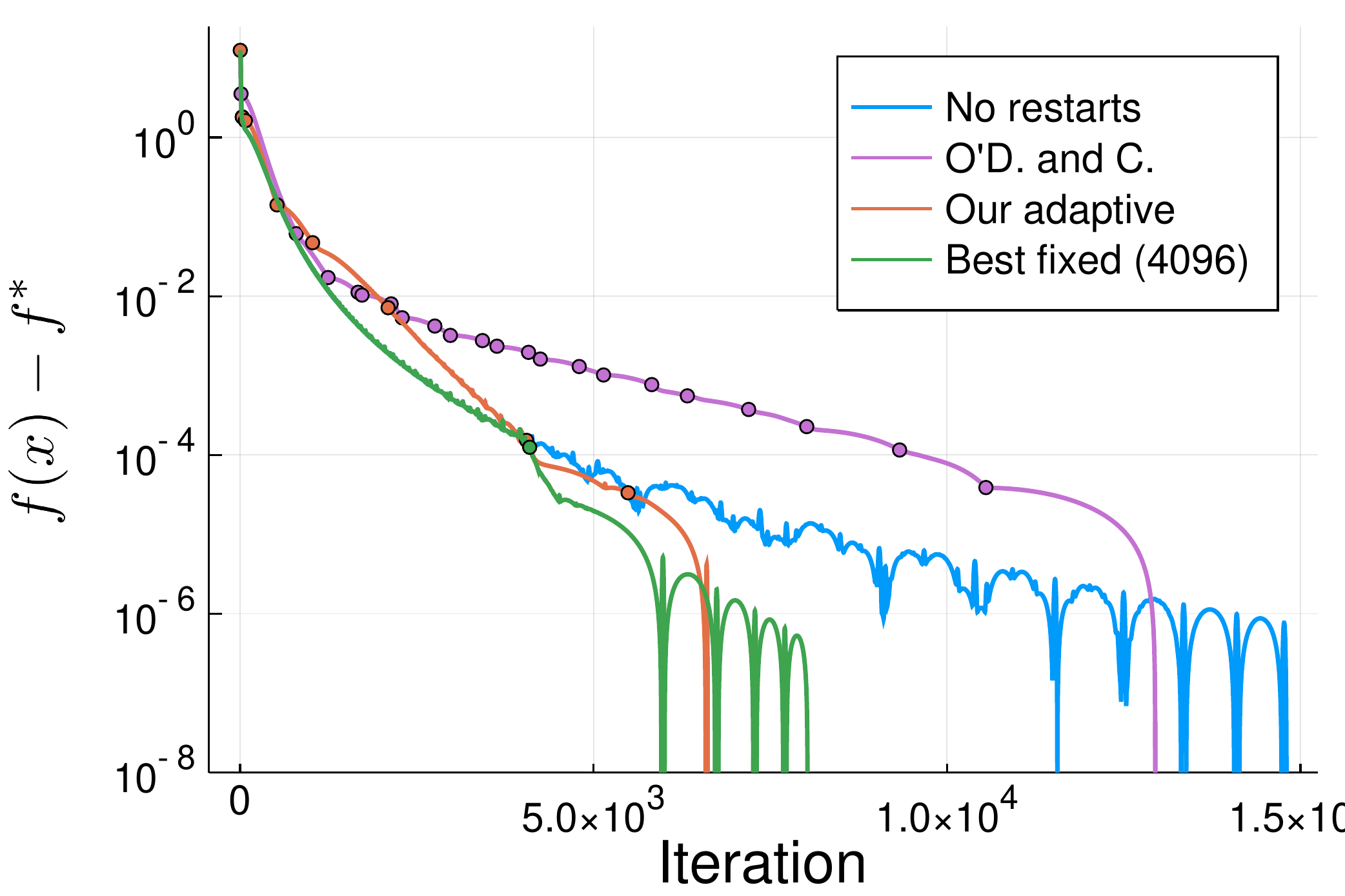}
\caption{Results for minimizing the function \eqref{eq:hard-example}.  The best fixed restart period is found via grid search on $\{ 128, 256, 512, 1024, 2048, 4096, 8192 \}$.}\label{figure:hard-example}
\end{figure}

\end{document}